\documentclass[11pt,a4paper]{article}
\usepackage[utf8]{inputenc}
\usepackage{amsmath}
\usepackage{amsfonts}
\usepackage{titlesec}
\usepackage[capposition=top]{floatrow}
\usepackage[margin=1.2in]{geometry} %changes the margin size
\usepackage{changepage}
\usepackage{setspace}
\usepackage{amssymb}
\usepackage{amsthm}
\usepackage{relsize}
\usepackage{nccmath}
\usepackage{marginnote}
\usepackage[symbol]{footmisc}
\usepackage{mathtools}
\usepackage{graphicx}
\usepackage{epstopdf}
\usepackage{etoolbox}
\usepackage{appendix}

\usepackage{titlesec}
\titleformat{\chapter}[display]   
{\large \bfseries}{\chaptertitlename\ \thechapter}{16pt}
{\normalsize \thechapter{.}\ }[]

\titleformat{\section}[display]   
{\scshape}{}{0.5ex}
{\thesection. \hsp \hsp \centering}[]

\makeatletter
\patchcmd{\thebibliography}{%
  \section*{References}}{}{}% use * so that the sec is not numbered.
\makeatother

\let\ppprod\prod
\let\sssum\sum
\usepackage{fouriernc} %this is the font

\renewcommand{\prod}{\mathlarger{\ppprod}} %makes product symbols larger
\renewcommand{\sum}{\mathlarger{\sssum}} %makes summation symbols larger

\newcommand{\lb}{\left[}
\newcommand{\rb}{\right]}
\newcommand{\lp}{\left(}
\newcommand{\rp}{\right)}
\newcommand{\lbr}{\left\lbrace}
\newcommand{\rbr}{\right\rbrace}
\newcommand{\nCr}[2]{{{#1}\choose{#2}}}
\newcommand{\hspm}{\hspace{-0.1cm}}
\newcommand{\hsp}{\hspace{0.1cm}}
\newcommand{\R}{{\bf R}}
\newcommand{\Z}{{\bf Z}}
\newcommand{\SL}{\mathrm{SL}}
\newcommand{\expect}{\mathbb{E}}
\numberwithin{equation}{section}
\allowdisplaybreaks

\newtheoremstyle{sltheorem}
{1.2em}                % Space above
{1.2em}                % Space below
{\slshape}        % Theorem body font % (default is "\upshape")
{}                % Indent amount
{\scshape}       % Theorem head font % (default is \mdseries)
{.}               % Punctuation after theorem head % default: no punctuation
{.5em}               % Space after theorem head
{}                % Theorem head spec
\theoremstyle{sltheorem}

\newtheorem{theorem}{Theorem}

\newtheorem{prop}[theorem]{Proposition}
\newtheorem{lemma}[theorem]{Lemma}
\numberwithin{theorem}{section}

\begin{document}
\thispagestyle{empty}
\begin{center}
\noindent
\textbf{ON THE DISTRIBUTION OF ANGLES BETWEEN \\ INCREASINGLY MANY SHORT LATTICE VECTORS}\\ \vspace{0.3cm}
\textsc{Kristian Holm} \quad \quad \today
\end{center} 
\vspace{0.4cm}
\begin{adjustwidth}{5em}{5em}
\begin{small}
\textbf{Abstract.} Following Södergren, we consider a collection of random variables on the space $X_n$ of unimodular lattices in dimension $n$: Normalizations of the angles between the $N = N(n)$ shortest vectors in a random unimodular lattice, and the volumes of spheres with radii equal to the lengths of these vectors. We investigate the expected values of certain functions evaluated at these random variables in the regime where $N$ tends to infinity with $n$ at the rate $N = \smash{o \lp n^{1/6} \rp}$. Our main result is that as $n \longrightarrow \infty$, these random variables exhibit a joint Poissonian and Gaussian behaviour. 
\end{small}
\end{adjustwidth}
%\vspace{0.4cm}
\section{Introduction}
Let $X_n$ be the space of $n$-dimensional unimodular lattices. As a homogeneous space, $X_n$ may be identified with the quotient $\SL (n, \R) / \SL (n, \Z)$, and therefore it inherits an $\SL(n, \R)$-invariant measure from the Haar measure on the special linear group. In spite of being non-compact, the quotient $\SL (n, \R) / \SL (n, \Z)$ has finite measure (see e.g. \cite[Thm. 7.0.1]{morris}), and hence, after normalizing, we obtain a probability measure $\mu_n$ on $X_n$ and can talk meaningfully about \textit{random lattices.} \par 
Let $L \in X_n$ and let $N \geqslant 2$ be an integer. If we let $\sim$ be the relation on $L$ given by $\textbf{v} \sim \textbf{w} \Leftrightarrow \textbf{v} = \pm \textbf{w}$,
we consider representatives $\textbf{v}_1, \ldots, \textbf{v}_N$ of the equivalence classes corresponding to the $N$ shortest vectors in $L$.  We then obtain a finite non-decreasing sequence of positive real numbers given by the lengths of the $\textbf{v}_i$, namely
\begin{align*}
0 < |\textbf{v}_1| \leqslant |\textbf{v}_2 | \leqslant \cdots \leqslant |\textbf{v}_N|.
\end{align*}
By considering these lengths as radii of Euclidean balls in $\R^n$, we obtain the sequence
\begin{align*}
0 < \mathcal{V}_1 (L) \leqslant \cdots \leqslant \mathcal{V}_N (L) \quad \quad \mathcal{V}_j(L) := \frac{\pi^{n/2}}{\Gamma (n/2 + 1)} |\textbf{v}_j|^n,
\end{align*}
so that $\mathcal{V}_j := \mathcal{V}_j (L)$ is the volume of an $n$-dimensional ball of radius $|\textbf{v}_j|$. The sequence $\lbrace \mathcal{V}_j \rbrace_{j \geqslant 1}$, and its first few elements in particular, encodes important geometric information about the lattice $L$. For example, the quantity
\begin{align*}
2^{-n} \sup_{L \in X_n} \left\lbrace \mathcal{V}_1(L) \right\rbrace
\end{align*}
determines the maximal density of an $n$-dimensional lattice sphere packing. Therefore it is an interesting question how $\lbrace \mathcal{V}_j (L) \rbrace_{j \geqslant 1}$ is distributed; in particular if it has any limiting distribution as $\dim L = n$ tends to infinity. The first result concerning this matter deals with the case $N = 1$ and is due to Rogers who proved \cite[Thm. 3]{rogers56} that $\mathcal{V}_1$ is exponentially distributed with mean $2$ in the limit as $n \longrightarrow \infty$. Södergren generalized this result by determining the limiting distribution for any fixed number $N$ of lattice vectors. His result \cite[Thm. 1]{sodapoisson} states that for fixed $N \geqslant 1$, the sequence $\lbrace \mathcal{V}_j \rbrace_{j \leqslant N}$ in fact converges in distribution to the first $N$ points of a Poisson process on $\R_+$ with intensity $\tfrac{1}{2}$. \par 

Aside from the lengths of the shortest vectors, we also want to consider the angles between them. Given that we are only counting one of the vectors $\textbf{v}_i$ and $-\textbf{v}_i$, however, the angles between the $\textbf{v}_i$ are not a priori well-defined. We circumvent this ambiguity by always taking the angle in the interval $\lb 0, \tfrac{\pi}{2} \rb$. More precisely, if $a(\textbf{x}, \textbf{y}) \in \lb 0, \pi \rb$ is the angle between two non-zero vectors $\textbf{x}, \textbf{y} \in \R^n$, we define 
\begin{align*}
\varphi_{ij} := \varphi_{ij}(L) := \varphi (\textbf{v}_i, \textbf{v}_j) := \begin{cases}
a(\textbf{v}_i, \textbf{v}_j ) \quad &\text{if } a(\textbf{v}_i, \textbf{v}_j ) \in \lb 0, \tfrac{\pi}{2} \rb,\\
\pi - a(\textbf{v}_i, \textbf{v}_j) \quad &\text{otherwise,}
\end{cases}
\end{align*}
for all \textit{admissible} pairs $(i,j)$, by which we understand a pair of integers $i$ and $j$ satisfying $1 \leqslant i < j \leqslant N$. As a consequence of the counter-intuitive fact that the volume of the unit sphere in high dimension is concentrated, to a large extent, around its "poles," it is known that in high dimension, the vectors $\textbf{v}_i$ are very close to being orthogonal. More precisely, for a fixed integer $N \geqslant 2$, 
\begin{align}\label{p1-massconcentration}
\mu_n \left( \left\lbrace L \in X_n : \tfrac{\pi}{2} - \varphi_{ij} > \tfrac{C}{\sqrt{n}} \text{ for some $i < j \leqslant N$} \right\rbrace \right) \longrightarrow 0
\end{align}
as $n, C \longrightarrow \infty$, cf. \cite[Prop. 1.1]{soda10}. This fact motivates the study of the \textit{normalized} random variables $\lbrace \tilde{\varphi}_{ij} \rbrace$,
\begin{align}\label{p1-disdanormalization}
\tilde{\varphi}_{ij} := \tilde{\varphi}_{ij} (L) := \sqrt{n}\left( \tfrac{\pi}{2} - \varphi_{ij} \right).
\end{align}
\par If we let $L$ vary over the space $X_n$, the $\mathcal{V}_j (L)$ ($1 \leqslant j \leqslant N$) and the $\tilde{\varphi}_{ij}(L)$ ($1 \leqslant i < j \leqslant N$) form a collection of $N + \smash{\nCr{N}{2}}$ random variables on the space of unimodular lattices. In \cite{soda10} Södergren studied the distribution of these in high dimension and proved the following theorem \cite[Thm. 1.2]{soda10}: If $N \geqslant 2$ is fixed, then the distribution of the random vector
\begin{align*}
\textbf{w}(L, N) := \left( \mathcal{V}_1, \ldots, \mathcal{V}_N, \tilde{\varphi}_{12}, \ldots, \tilde{\varphi}_{(N-1)N} \right)
\end{align*}
converges as $n \longrightarrow \infty$ to the joint distribution of the first $N$ points of a homogeneous Poisson process on $\R_+$ with intensity $\tfrac{1}{2}$ and a collection of $\nCr{N}{2}$ independent positive standard Gaussians that are also independent of the firsts $N$ random variables.\\\\
\textsc{Remark.} Here and throughout the paper, a "positive standard Gaussian" is to be understood in the sense of a random variable distributed like $|X|$ for $X \sim N(0,1)$.\\
\par Södergren's theorem marks the starting point of the present investigation. Intuitively, we want to investigate the distributional limit when we allow for more and more of the shortest lattice vectors to be taken into account as the dimension $n$ grows. In more detail, our question regards the existence and nature of a "limiting distribution" as $n \longrightarrow \infty$ of $\textbf{w}(L, N)$ if we allow $N$ to \textit{grow as a small power of} $n$. Quotation marks seem to be in order here, for the random vector takes values in a space that changes according with the parameter that tends to infinity. More precisely, we want to investigate the following: \textit{If we take $N = N(n) \sim n^\delta$, for suitable $\delta > 0$, can we still "observe" the joint Gaussian and Poissonian behaviour of the random vector} $\textbf{w}(L, N)$ \textit{as $n \longrightarrow \infty$?} 
\par Of course, we should be more explicit about what we mean by "observe." To clarify this, suppose that $0 < T_1 < T_2 < \cdots$ are the points of a homogeneous Poisson process on $\R_+$ with intensity $\frac{1}{2}$ and let the random variables $\Phi_{ij}$ (for $1 \leqslant i < j \leqslant N$) be positive standard Gaussians. Suppose also that $f$ is a a compactly supported and measurable function, possibly with additional properties as well. Then we are looking for an asymptotic estimate of the form 
\begin{align}\label{p1-content}
\expect \lb f\left(\mathcal{V}_1, \ldots, \mathcal{V}_N, \tilde{\varphi}_{12}, \ldots, \tilde{\varphi}_{(N-1)N}  \right) \rb = \expect \lb f \left( T_1, \ldots, T_N, \Phi_{12}, \ldots, \Phi_{(N-1)N} \right) \rb + \text{error}
\end{align}
for some explicit error term. Note that in the case where $N$ is fixed relative to $n$, if (\ref{p1-content}) holds for all continuous $f$ with compact support with an error that tends to $0$ as $n \longrightarrow \infty$, one knows that $\textbf{w}(L,N)$ converges in distribution to $\lp T_1, \ldots, T_N, \Phi_{12}, \ldots, \Phi_{N-1,N} \rp$, cf. \cite[Thm. 5.3]{bhat+way}. The fact that such an asymptotic estimate holds in the present context is the main theorem of this paper. 
\begin{theorem}\label{p1-maintheorem}
Let $N =N(n) \geqslant 2$ be an integer with $N \longrightarrow \infty$ as $n \longrightarrow \infty$ and $N = o\lp n^{1/6} \rp$. Write $b(N) = \nCr {N}{2}$, and let $K \geqslant 3$ and $M > 0$. Suppose that
\begin{align*}
f_{0,N} : \lp \R_{\geqslant 0} \rp^{N} \longrightarrow \R_{\geqslant 0}, \quad \quad f_{1,N} : \lp \R_{\geqslant 0} \rp^{b(N)} \longrightarrow \R_{\geqslant 0}
\end{align*}
are Borel measurable functions with $\int f_{1,N} > 0$, satisfying the properties
\begin{align*}
\left\Vert f_{0,N} \right\Vert_\infty , \left\Vert f_{1,N} \right\Vert_\infty \leqslant \sqrt{M}, \quad \quad \mathrm{supp} \lp f_{0,N} \rp \subset \lb 0 , K \rb^N, \quad \quad \mathrm{supp} \lp f_{1,N} \rp \subset \lb 0, K \rb^{b(N)}.
\end{align*}
Define $f_N \lp x_1, \ldots, x_N, x_{N+1}, \ldots, x_{N + b(N)} \rp := f_{0,N} \lp x_1, \ldots, x_N \rp f_{1,N} \lp x_{N+1}, \ldots, x_{N + b(N)} \rp$ for all non-negative real numbers $x_1, \ldots, x_{N + b(N)}$. Finally, let 
\begin{align*}
G_f (N) = \expect \lb f_{1,N} \lp \Phi_{12}, \ldots, \Phi_{N-1,N} \rp \rb.
\end{align*}
Then, as $n \longrightarrow \infty$, we have
\begin{align}\label{p1-statementofmain}
\nonumber \expect \lb f_N \left(\mathcal{V}_1, \ldots, \mathcal{V}_N, \tilde{\varphi}_{12}, \ldots, \tilde{\varphi}_{N-1,N}\right) \rb &= \lp 1 + O_K \lp \frac{N^3}{\sqrt{n}} \rp \rp \expect \lb f_N \left(T_1, \ldots, T_N, \Phi_{12}, \ldots, \Phi_{N-1,N} \right) \rb   \\ &\quad \quad + O_{K,M} \left( 2^{-N} G_f(N) \left( \frac{2 e K}{N^2} \right)^{N^2} + \lp \frac{\sqrt{3}}{2} \right)^n \hspace{-0.1cm} K^{4N^4} \right).
\end{align}
\end{theorem}
\noindent \textsc{Remarks.} \par 1) An example of an explicit class of functions where the statement of Theorem \ref{p1-maintheorem} is meaningful, in the sense that the intended main term dominates, is the following: $f_{0,N}$ is the indicator function of the set $\lb 0, K \rb^{N}$, and $f_{1,N}$ is the indicator function of the Cartesian product $D_{12} \times \cdots \times D_{N-1,N}$ where $D_{ij} \subset \lb 0, K \rb$ is an arbitrary measurable set of Lebesgue measure at least $\xi > 0$, for each $(i,j)$. We elaborate on this in Section 4. \par
2) The proof of Theorem \ref{p1-maintheorem} follows Södergren's proof of \cite[Thm. 1.2]{soda10} and proceeds by rewriting all limits in his argument as effective asymptotic statements. It is not possible to improve on the allowed growth rate of $N$ with this method if one wants asymptotic statements where the intended main term dominates. \par
3) One could ask whether the normalization (\ref{p1-disdanormalization}) is still the most natural one to use in the situation where $N$ increases with $n$. That is, does a version of (\ref{p1-massconcentration}) still hold in the regime $N = o \lp n^{1/6} \rp$? We study this question in Section 5 and prove an version of (\ref{p1-massconcentration}) in the form of Proposition \ref{p1-anglenormalization-prop}. Although $C$ has to slowly tend to infinity with $N$ in this result, the proposition indicates that (\ref{p1-massconcentration}) is still a natural normalization in the current situation. \par 
4) If we take $f_N \leqslant M$ to be an arbitrary non-negative function supported on $\lb 0, K \rb^{N + b(N)}$, unfortunately we do not obtain an error term of the same quality as in Theorem \ref{p1-maintheorem}. In this case, we can prove that 
\begin{align}\label{p1-oldthm4}
\nonumber \expect \lb f_N (\mathcal{V}_1, \ldots, \mathcal{V}_N, \tilde{\varphi}_{12}, \ldots, \tilde{\varphi}_{N-1,N}) \rb &= \expect \lb f_N (T_1, \ldots, T_N, \Phi_{12}, \ldots, \Phi_{N-1,N}) \rb \\ &\quad \quad + O_{K,M} \left( \frac{ P_L(K)^{\nCr{N}{2}}(K/2)^N N^3}{N! \sqrt{n}} \right)
\end{align}
where $P_L$ denotes the left cumulative distribution function of the positive standard Gaussian distribution $\left| N(0,1) \right|$. The proof of this is identical to that of Theorem \ref{p1-maintheorem}, except that only a weaker version of the important auxiliary result Theorem \ref{p1-maintechnical} holds and can be applied. (See below for a brief outline of the proof of Theorem \ref{p1-maintheorem}.) Regrettably, the only case in which the main term in (\ref{p1-oldthm4}) dominates is the one where $f$ has full support in its last $\nCr{N}{2}$ variables.\par 
If we drop the condition that $f_N$ should be non-negative, but maintain that $f_N$ should satisfy the factorization property $f = f_{0,N} f_{1,N}$ mentioned in the theorem, it is possible that the main term on the right-hand side of (\ref{p1-statementofmain}) vanishes, at least unless other conditions on $f_N$ are imposed, even though the error terms may be non-zero. Moreover, the part of the proof of Theorem \ref{p1-maintechnical} that gives the better error term in (\ref{p1-statementofmain}), compared to (\ref{p1-oldthm4}), requires that, at the very least, $f_{1,N}$ should be non-negative or non-positive. \par 
5) As a first step towards Theorem \ref{p1-maintheorem}, we first examined the question of the distribution of normalized angles between $N$ random points on the sphere $S^{n-1}$ when $N$ tends to infinity with $n$. We obtained the result that if $N = o \lp n^{1/6} \rp$, one can still observe the standard Gaussian distribution among the normalized angles as in the case where $N$ is fixed compared to $n$. We prove this result in the form of Theorem \ref{p1-pointsonspheretheorem} in Section 2.\\ \par 
Theorem \ref{p1-maintheorem} is by no means the first result of its kind, in that the problem of examining the limiting distribution of (a subset of) the random variables $\mathcal{V}_j$ $(1 \leqslant j \leqslant N)$ and $\tilde{\phi}_{ij}$ $(1 \leqslant i < j \leqslant N)$, with $N$ depending on $n$, has been studied before on several occasions. For example, Kim proved \cite[Thm. 4]{kim16} that the Poissonian behaviour in $\mathcal{V}_1, \ldots, \mathcal{V}_N$, which Södergren observed in the case of a fixed $N$, prevails if, for an arbitrary $\varepsilon > 0$, $N \leqslant \lp \tfrac{n}{2} \rp^{1/2 - \varepsilon}$. This was later generalized by Strömbergsson and Södergren who proved \cite[Thm. 1.2]{stromsod} that the Poissonian behaviour prevails even in the subexponential regime $N = O_\varepsilon (e^{\varepsilon n})$ for any $\varepsilon > 0$. 
\par In spite of this, unfortunately it seems far too optimistic to hope for any results that involve a limiting distribution of the angles if one takes $N \sim n^{\delta}$ and $\delta > 1$. Aside from certain technical limitations in the use of Rogers' integration formula (see below), if $n = o \lp N \rp$, then by linear algebra, the increasing number of vectors under consideration ensures increasingly many linear dependencies between the short lattice vectors, cf. \cite[Prop. 3.1]{kim18}. Heuristically, then, these dependencies would interfere with the random behaviour observed for large $n$ in the case of a fixed number $N$ of lattice vectors. However, another theorem of Kim shows that it is possible to take $\delta$ very close to $1$ if one is content with proving a slightly different statement about the limiting distribution of the normalized angles. Concretely, if $U$ is the volume of the unit ball in $\R^n$ and $u$ is a point on $S^{n-1}$ chosen from a uniform distribution, he shows \cite[Thm. 1.4]{kim18} that with $N = o(n / \log n)$, which amounts to letting $\delta$ be any positive number satisfying
\begin{align*}
\delta = \delta(n) < 1 - \frac{\log \log n}{\log n},
\end{align*}
the two point processes 
\begin{align*}
\bigl( \left| \textbf{v}_i \right| , \pm  \textbf{v}_i / \left| \textbf{v}_i \right| \bigr), \quad \quad \lp \lp T_i / U \rp^{1/n}, \pm u \rp  \quad \quad (i = 1, \ldots, N)
\end{align*}
are asymptotically equal as $n \longrightarrow \infty$. In the case of a \textit{fixed} $N$, this statement is no less precise than what Södergren proves, since then the limiting distribution (as $n \longrightarrow \infty$) of the normalized angles between $N$ random points on $S^{n-1}$ is precisely equal to the joint distribution of $\nCr{N}{2}$ independent standard Gaussians, cf. \cite[Thm. 3.2]{soda10}. However, if one allows $N$ to grow with $n$, it is not clear whether these two distributions coincide. Therefore Kim's result can be viewed as a substantial widening of the allowance of the asymptotic behaviour of $N$ relative to $n$, obtained at a slight expense of precision in the statement about the limiting distribution of the normalized angles.\par 
The method used by Kim also differs from the method of Södergren. What motivated our investigations was the question of whether the latter method can be applied to yield a result where $N$ is allowed to grow with $n$, and where convergence of $\textbf{w}(L,N)$ to a jointly Poissonian and Gaussian distribution is proved, in the sense of (\ref{p1-content}). 
\par We now give a very brief outline of the structure of the proof of Theorem \ref{p1-maintheorem}. The main technical tool needed in the proof is an integration formula \cite[Thm. 4]{rogers55} due to Rogers. This formula can be expressed as follows. Let $\rho : \lp \R^n \rp^k \longrightarrow \R$ be a non-negative measurable function. For a lattice $L \in X_n$, let $L' = L \setminus \lbrace 0 \rbrace$ denote the punctured lattice obtained from $L$ by removing the origin. Then the integral
\begin{align*}
\int_{X_n} \sum_{\textbf{m}_i \in L' \atop (1 \leqslant i \leqslant k) } \rho \lp \textbf{m}_1, \ldots, \textbf{m}_k \rp \hsp d \mu_n (L)
\end{align*}
equals
\begin{align}\label{p1-rogersformula}
\nonumber &\int_{\R^n} \cdots \int_{\R^n} \rho \lp \textbf{x}_1, \ldots, \textbf{x}_k \rp \hsp d\textbf{x}_1 \cdots d\textbf{x}_k \\ &\quad \quad + \sum_{(\nu, \mu)} \sum_{q \geqslant 1} \sum_D \lp \frac{e_1}{q} \cdots \frac{e_m}{q} \rp^n \int_{\R^n} \cdots \int_{\R^n} \rho \lp \sum_{i = 1}^m \frac{d_{i1}}{q} \textbf{x}_i , \ldots, \sum_{i = 1}^m \frac{d_{ik}}{q}\textbf{x}_i \rp \hsp d \textbf{x}_1 \cdots d\textbf{x}_m.
\end{align}
In the last expression, the outermost sum is over all partitions
\begin{align*}
\lp \nu , \mu \rp = \lp \nu_1, \ldots, \nu_m, \mu_1, \ldots, \mu_{k-m} \rp
\end{align*}
of $\lbrace 1, \ldots, k \rbrace$ into two non-empty increasing sequences (that is, one has $1 \leqslant m \leqslant k-1$). The innermost sum is over all integer matrices $D$ all of whose entries have greatest common divisor $1$, and whose columns are all non-zero, and such that for $i,j = 1, \ldots, m$, one has $d_{i \nu_j} = q \delta_{ij}$, and such that if $\mu_j < \nu_i$ for any $i = 1, \ldots, m, \, j = 1, \ldots, k-m$, then $d_{i \mu_j} = 0$. Finally, for all $i$, $e_i = \gcd \lp q, \varepsilon_i \rp$ with $\varepsilon_i$ the $i$'th elementary divisor of the matrix $D$. \par 
Our concrete use of Rogers' formula is that it will be used to give an effective estimate of the expected value
\begin{align*}
\expect \lb \sum_{\textbf{n} \in M_\lambda} f_N \lp \mathcal{V}_{n_1}, \ldots, \mathcal{V}_{n_\lambda}, \tilde{\varphi}_{n_1 n_2}, \ldots, \tilde{\varphi}_{n_{N-1},n_N} \rp \rb,
\end{align*}
where $\lambda = N + \ell$ for a non-negative integer $\ell$, and the summation takes place over all $\lambda$-tuples $\textbf{n} = \lp n_1, \ldots, n_\lambda \rp$ of positive integers with pairwise distinct entries. Such an estimate is given by Theorem 3.1 in Section 3. In Section 5, we then use a sieving technique used by Södergren in \cite{soda10} to obtain from Theorem 3.1 an effective estimate of 
\begin{align*}
\expect \lb f_N \lp \mathcal{V}_1, \ldots, \mathcal{V}_N, \tilde{\varphi}_{12}, \ldots, \tilde{\varphi}_{N-1,N} \rp \rb
\end{align*}
and prove Theorem \ref{p1-maintheorem}. \newpage
\begin{center}
\textsc{Notation}
\end{center}
We give here a list of special notation used in the paper which the reader may refer back to. \\
\begin{flushleft}
$I(A)$ is the indicator function of the set or statement $A$. \vspace{0.1cm} \\
$L$ is a (Euclidean) unimodular lattice of dimension $n$. \\ \vspace{0.1cm}
$M_\lambda$ ($\lambda \geqslant 2$ an integer) is the set of $\lambda$-tuples of integers, all at least $1$, with pairwise distinct entries.\\ \vspace{0.1cm}
$N$ is the number of short vectors in $L$ under consideration.\\ \vspace{0.1cm}
$N_n (L, x) = \# \left\lbrace j : \mathcal{V}_j \leqslant x \right\rbrace$ is the number of lattice vectors in $L$ with length at most equal to the radius of the $n$-dimensional sphere of volume $x$.\\ \vspace{0.1cm}
$N_\infty ( \lb a, b \rb) = \# \left\lbrace j : a \leqslant T_j \leqslant b \right\rbrace$ is the number of points of the Poisson process $\lbrace T_i \rbrace$ that lie between $a$ and $b$.\\ \vspace{0.1cm}
$N_\infty (x) = N_\infty \bigl( \lb 0, x \rb \bigr)$ is the number of points of the Poisson process $\lbrace T_i \rbrace$ that are at most equal to $x$.\\ \vspace{0.1cm}
$\mathbb{P} ( A )$ will denote the probability (or measure) of the statement (or set) $A$. We stress that $\mathbb{P}$ will always refer to the "natural" probability measure on the space that $A$ belongs to so that different instances of "$\mathbb{P}$" may refer to different measures. \\ \vspace{0.1cm}
$\lbrace T_i \rbrace_{i \geqslant 1}$ is an increasing sequence of points in a homogeneous Poisson process of intensity $1/2$.\\ \vspace{0.1cm}
$\textbf{v}_i$ is a representative (modulo $\pm$) of the $i$'th shortest vector in the lattice $L$.\\ \vspace{0.1cm}
$\mathcal{V}_{i}$ is the volume of the $n$-dimensional sphere of radius equal to the length $\left| \textbf{v}_i \right|$ of the $i$'th shortest vector in $L$.\\ \vspace{0.1cm}
$X_n$ is the space of (Euclidean) unimodular lattices of dimension $n$, identified with the homogeneous space $\mathrm{SL}(n, \R) / \mathrm{SL}(n, \Z)$.\\ \vspace{0.1cm}
$\lbrace \Phi_{ij} : 1 \leqslant i < j \leqslant N \rbrace$ is a collection of independent, positive Gaussians.
\end{flushleft}
\section{Angles between Random Points on the $n$-Sphere}
Towards a proof of Theorem \ref{p1-maintheorem}, the question of the distribution of certain normalizations of the angles between $N = N(n)$ random points on $S^{n-1}$ seemed to be a fitting \textit{toy problem}, progress on which would hopefully serve as a decent measure of the chances of success in employing Södergren's method to prove an effective statement of the form (\ref{p1-content}). In connection with this study, we obtained the following result.
\begin{theorem}\label{p1-pointsonspheretheorem}
\textit{Let $N = o\left(n^{1/6}\right)$ be an integer. Let $K > 0$ be a real number. Let $\varepsilon > 0$ be any positive number, and assume that for $1 \leqslant i < j < \infty$ we have numbers $c_{ij}$ and $c_{ij}'$ such that for all $i < j$,}
\begin{align*}
-K \leqslant c_{ij} < c_{ij}' \leqslant K, \quad \quad \inf_{i < j} \hspace{0.05cm} \left\lbrace c_{ij}' - c_{ij} \right\rbrace \geqslant \varepsilon.
\end{align*}
\textit{Let $I(t_{ij})$ be short for $I\left(c_{ij} < t_{ij} \leqslant c_{ij}'\right)$ where $I$ denotes a characteristic function. Let}
\begin{align*}
G(N) &= \left( \frac{1}{\sqrt{2 \pi}} \right)^{{{N}\choose{2}}} \int_{-K}^K \cdots \int_{-K}^K \prod_{i,j = 1 \atop i < j}^N I(t_{ij}) e^{-t_{ij}^2 / 2} \hspace{0.1cm} d t_{(N-1)N} \cdots dt_{12},
\end{align*}
\textit{so that $G(N)$ is the probability that each Gaussian $X_{ij}$ in a collection $\lbrace X_{ij} : 1 \leqslant i < j \leqslant N \rbrace$ of independent standard Gaussians lies between $c_{ij}$ and $c_{ij}'$. Let}
\begin{align*}
P(N) &= \sigma_{n-1}^N \left( \left\lbrace (\textbf{u}_1, \ldots, \textbf{u}_N) \in (S^{n-1})^N : c_{ij} < \tilde{\alpha}_{ij}  \leqslant c_{ij}' \text{ for } 1 \leqslant i < j \leqslant N \right\rbrace \right)
\end{align*}
\textit{where } $\tilde{\alpha}_{ij} = \sqrt{n} \left( a(\textbf{u}_i, \textbf{u}_j) - \tfrac{\pi}{2} \right)$ \textit{and $\sigma_{n-1}^N$ denotes the $N$-fold product measure of the $(n-1)$-dimensional probability measure on $S^{n-1}$. Then as $n \longrightarrow \infty$, }
\begin{align*}
\frac{P(N)}{G(N)} =  1 + O_{K, \varepsilon} \left( \frac{N^3}{\sqrt{n}} \right). 
\end{align*}
\end{theorem}
To prove Theorem \ref{p1-pointsonspheretheorem}, we will need several lemmas giving effective asymptotic statements about the individual factors in an explicit expression for the probability $P(N)$.
\begin{lemma}\label{p1-lemma2.2}
\textit{Let $N = o \lp n^{1/3} \rp$ be an integer. For $1 \leqslant i < j \leqslant N$, let $t_{ij} \in \lb -K-1, K+1 \rb$ where $K > 0$. Then, as $n \longrightarrow \infty$, we have}
\begin{align}\label{p1-lemma}
\prod_{i,j=1 \atop i < j}^{N} \cos \left( \frac{t_{ij}}{\sqrt{n}} \right)^{n-i-1} = \left( 1 + O_K \left( \frac{N^3}{n} \right) \right) \prod_{i,j = 1 \atop i<j}^{N} e^{-t_{ij}^2/2}.
\end{align}
\end{lemma}
\begin{proof}
Let $n > 2(K+1)^2$. By elementary Taylor expansions, we see that as $n \longrightarrow \infty$,
\begin{align}\label{p1-logcos}
\log \lp \cos \frac{t_{ij}}{\sqrt{n}} \rp = - \frac{t_{ij}^2}{2n} + O_K \left( \frac{1}{n^2} \right).
\end{align}
Taking the logarithm of the left-hand side of (\ref{p1-lemma}) and applying (\ref{p1-logcos}), we find that
\begin{align*}
\log \prod_{i,j=1 \atop i < j}^{N} \cos \left( \frac{t_{ij}}{\sqrt{n}} \right)^{n-i-1} &= \sum_{i,j = 1 \atop i < j}^N (n-i-1) \cdot \log \lp \cos \frac{t_{ij}}{\sqrt{n}} \rp \\
&= \sum_{i,j = 1 \atop i < j}^N (n-i-1) \cdot \left( - \frac{t_{ij}^2}{2n} + O_K \left( \frac{1}{n^2} \right) \right) \\
&= \sum_{i,j = 1 \atop i < j}^N \left[ -\frac{t_{ij}^2}{2} + O_K \left( \frac{1}{n} \right) + (i+1) \cdot \left( \frac{t_{ij}^2}{2n} + O_K \left( \frac{1}{n^2} \right) \right) \right] \\
&=  - \hspm  \sum_{i,j = 1 \atop i<j}^N \frac{t_{ij}^2}{2} + {{N}\choose{2}} \cdot O_K \left( \frac{1}{n} \right)+ 
O_K \lp \frac{1}{n} \rp \sum_{i, j = 1 \atop i < j}^N \lp i + 1 \rp \\
&=  - \hspm \sum_{i,j = 1 \atop i<j}^N \frac{t_{ij}^2}{2} + O_K \left( \frac{N^2}{n} \right) + O_K \left( \frac{1}{n} \right) \cdot \sum_{i = 1}^N (N-i)(i+1).
\end{align*}
Since $\sum_{i = 1}^N (N-i)(i+1) = \tfrac{1}{6} (N^3 + 3N^2 - 4N) = O(N^3)$, we now have
\begin{align}\label{p1-partial}
\log \prod_{i,j=1 \atop i < j}^{N} \cos \left( \frac{t_{ij}}{\sqrt{n}} \right)^{n-i-1} = - \hspm \sum_{i,j = 1 \atop i<j}^N \frac{t_{ij}^2}{2} + O_K \left( \frac{N^3}{n} \right).
\end{align}
By exponentiating (\ref{p1-partial}), we finally obtain
\begin{align*}
\prod_{i,j = 1 \atop i < j}^N \cos \left( \frac{t_{ij}}{\sqrt{n}} \right)^{n-i-1} &= \left(1 + O_K \left(\frac{N^3}{n} \right) \right) \prod_{i,j = 1 \atop i < j}^N e^{-t_{ij}^2 / 2},
\end{align*}
which is (\ref{p1-lemma}). 
\end{proof}
\begin{lemma}\label{p1-lemma2.3}
\textit{Let $N = o \lp n^{1/3} \rp$ be an integer and let $\omega_k$ be the $(k-1)$-dimensional volume of $S^{k-1}$. Then }
\begin{align*}
\prod_{l = 1}^{N-1} \prod_{m = 1}^l \frac{\omega_{n-m}}{\omega_{n-m+1} \sqrt{n}} = \left( \frac{1}{\sqrt{2 \pi}} \right)^{{N}\choose{2}} \bigl( 1 + p \lp n, N \rp \bigr)
\end{align*}
for some function $p \lp n, N \rp = O \lp N^3 / n \rp$.
\end{lemma}
\begin{proof}
Stirling's formula gives
\begin{align*}
\frac{\omega_k}{\omega_{k+1}} &= \frac{1}{\sqrt{\pi}} \cdot \Gamma \left(\frac{k+1}{2} \right) \cdot \Gamma \left( \frac{k}{2} \right)^{-1} = \frac{1}{\sqrt{2\pi e}} \cdot \sqrt{k} \cdot \left( 1 + \frac{1}{k} \right)^{k/2} \cdot \left(1 + O\left( \frac{1}{k} \right) \right).
\end{align*}
Using the fact that $(1+1/k)^k = e + O(k^{-1})$ and the estimate
\begin{align*}
\log \lp \lp e + O \lp \frac{1}{k} \rp \rp^{1/2} \rp = \frac{1}{2} \lb 1 + \log \lp 1 + O \lp \frac{1}{k} \rp \rp \rb = \frac{1}{2} + O \lp \frac{1}{k} \rp,
\end{align*}
we obtain that
\begin{align}\label{p1-quotient}
\frac{\omega_k}{\omega_{k+1}} = \frac{1}{\sqrt{2\pi e}}  \cdot \sqrt{k} \cdot \left(e + O\left( \frac{1}{k} \right) \right)^{1/2} \cdot \left(1 + O\left( \frac{1}{k} \right) \right) = \frac{1}{\sqrt{2\pi}}  \cdot \sqrt{k} \cdot \left(1 + O\left( \frac{1}{k} \right) \right). 
\end{align}
Applying (\ref{p1-quotient}) with $k = n-m$, we now get
\begin{align*}
\frac{1}{\sqrt{n}} \cdot \frac{\omega_{n-m}}{\omega_{n-m + 1}} = \frac{1}{\sqrt{2 \pi}} \cdot \left( 1 + O \left( \frac{m}{n} \right) \right) \cdot \left(1 + O \left( \frac{1}{n-m} \right) \right) = \frac{1}{\sqrt{2 \pi}} + O \left( \frac{m}{n} \right).
\end{align*}
Finally, since $m < N$, we have
\begin{align*}
\prod_{l = 1}^{N-1} \prod_{m = 1}^l \frac{\omega_{n-m}}{\omega_{n-m+1} \sqrt{n}} &=  \prod_{l = 1}^{N-1} \prod_{m = 1}^l \left( \frac{1}{\sqrt{2 \pi}} + O \left( \frac{m}{n} \right) \right) \\
&= \left( \frac{1}{\sqrt{2 \pi}} \right)^{{N}\choose{2}} \cdot \left( 1 + O \left( \frac{N}{n} \right) \right)^{{N}\choose{2}}\\
&= \left( \frac{1}{\sqrt{2 \pi}} \right)^{{N}\choose{2}} \cdot \exp \left( {{N}\choose{2}} \cdot \log \left( 1 + O \left( \frac{N}{n} \right) \right) \right)\\
&= \left( \frac{1}{\sqrt{2 \pi}} \right)^{{N}\choose{2}} \cdot  \left( 1 + O \left(\frac{N^3}{n} \right) \right).
\end{align*}
With 
\begin{align*}
p (n, N) = -1 + \bigl( 2 \pi \bigr)^{\frac{1}{2}\nCr{N}{2}} \prod_{l = 1}^{N-1} \prod_{m = 1}^l \frac{\omega_{n-m}}{\omega_{n-m+1} \sqrt{n}},
\end{align*}
this proves the desired estimate.
\end{proof}
Anticipating the proof of Theorem \ref{p1-pointsonspheretheorem}, our strategy (which is based on \cite{soda10}) will be to use the spherical symmetry of the set
\begin{align*}
\left\lbrace (\textbf{u}_1, \ldots, \textbf{u}_N) \in \lp S^{n-1} \rp^N : c_{ij} < \tilde{\alpha}_{ij}  \leqslant c_{ij}' \text{ for } 1 \leqslant i < j \leqslant N \right\rbrace
\end{align*}
to compute the probability $P(N)$. Concretely, by employing a change of variables in an integral expression for $P(N)$, we will only need to look at a special case where the vectors $\textbf{u}_1, \ldots, \textbf{u}_N$ have a particularly simple form. Before we begin the proof of Theorem \ref{p1-pointsonspheretheorem}, we will therefore analyse this special case.
\begin{lemma}\label{p1-lemma2.4}
\textit{Let $N = o \lp n^{1/3} \rp$ be an integer. Consider the unit vectors}
\begin{align*}
\textbf{u}_1 &= (1, 0 , \ldots, 0),\\
\textbf{u}_ 2 &= (\cos \phi_{12}, \sin \phi_{12}, 0, \ldots, 0), \\
\textbf{u}_3 &= (\cos \phi_{13}, \sin \phi_{13} \cos \phi_{23}, \sin \phi_{13} \sin \phi_{23}, 0, \ldots, 0), \\
&\hspace{0.15cm} \vdots \\
\textbf{u}_N &= (\cos \phi_{1N}, \sin \phi_{1N} \cos \phi_{2N}, \ldots, \sin \phi_{1N} \cdots \sin \phi_{(N-1)N}, 0 ,\ldots, 0),
\end{align*}
\textit{where $\phi_{ij} \in \lb 0, \pi \rb$ for all admissible $(i,j)$, and let $\alpha_{ij} = \arccos ( \textbf{u}_i \cdot \textbf{u}_j)$ and $\tilde{\alpha}_{ij} = \sqrt{n} \lp \alpha_{ij} - \tfrac{\pi}{2} \rp$. Let $K > 0$, let $t_{ij} = \sqrt{n}\lp \phi_{ij} - \tfrac{\pi}{2} \rp$, and suppose that for all admissible $(i,j)$, $t_{ij} \in \lb -K-1 , K+1 \rb$.} Then, for any admissible $(i,j)$ there is a function $g_{ij} \lp \tilde{\alpha}_{12}, \ldots, \tilde{\alpha}_{N-1,N} \rp$, which depends on $n$ and $N$ and satisfies $g_{ij} = O_K (N / \sqrt{n})$, with 
\begin{align}\label{p1-dotprod}
t_{ij} = \tilde{\alpha}_{ij} + g_{ij} \lp \tilde{\alpha}_{12}, \ldots, \tilde{\alpha}_{N-1,N} \rp.
\end{align}
\end{lemma}
\begin{proof} Let $n > 2(K+1)^2$. By a computation, we first observe that
\begin{align}\label{p1-thedotproductfromhell}
\textbf{u}_i \cdot \textbf{u}_j 
= 
\sum_{m=1}^{i-1} \sin \frac{t_{mi}}{\sqrt{n}} \cdot \sin \frac{t_{mj}}{\sqrt{n}} \cdot \prod_{k=1}^{m-1} \cos \frac{t_{ki}}{\sqrt{n}} \cdot \cos \frac{t_{kj}}{\sqrt{n}} - \sin \frac{t_{ij}}{\sqrt{n}} \cdot \prod_{k = 1}^{i-1} \cos \frac{t_{ki}}{\sqrt{n}} \cdot \cos \frac{t_{kj}}{\sqrt{n}}.
\end{align}
Then, by using the estimates $\sin x = x + O \lp x^3 \rp$ and $\cos x = 1 + O \lp x^2 \rp$, we find that for any admissible $(i,j)$,
\begin{align*}
\textbf{u}_i \cdot \textbf{u}_j &= - \sin \frac{t_{ij}}{\sqrt{n}} \cdot \prod_{k = 1}^{i-1} \left(1 + O_K \left( \frac{1}{n} \right) \right) + \sum_{m=1}^{i-1} \left( \frac{t_{mi}}{\sqrt{n}} + O_K \left( n^{-3/2} \right) \right) \left( \frac{t_{mj}}{\sqrt{n}} + O_K \left( n^{-3/2} \right) \right) \\ &\quad \quad \times \prod_{k=1}^{m-1}  \left(1 + O_K \left( \frac{1}{n} \right) \right)
\\  &=  -  \sin \frac{t_{ij}}{\sqrt{n}} \cdot  \left(1 + O_K \left( \frac{1}{n} \right) \right)^{i-1} +O_K \left( \frac{1}{n} \right) \sum_{m=1}^{i-1}  \left(1 + O_K \left( \frac{1}{n} \right) \right)^{m-1}
\\  &= -  \sin \frac{t_{ij}}{\sqrt{n}} \cdot  \left(1 + O_K \left( \frac{1}{n} \right) \right)^{i-1} + (i-1) \cdot O_K \left( \frac{1}{n} \right) \cdot  \max \left\lbrace 1, \left(1 + O_K \left( \frac{1}{n} \right) \right)^{i-1} \right\rbrace .
\end{align*}
By Taylor expansions, we have
\begin{align*}
 \left(1 + O_K \left( \frac{1}{n} \right) \right)^{i-1} &= \exp \left( (i-1) \cdot \log  \left(1 + O_K \left( \frac{1}{n} \right) \right) \right) = \exp \left( O_K \left( \frac{i-1}{n} \right) \right) = 1 + O_K \left( \frac{i}{n} \right).
\end{align*}
Let us assume that the constant implied by this notation is positive, as the alternative would lead us to an entirely analogous computation with an identical result. Then, by the previous computation we obtain
\begin{align}\label{p1-lolol}
\nonumber \textbf{u}_i \cdot \textbf{u}_j &= -  \sin \frac{t_{ij}}{\sqrt{n}} \cdot \left(1 + O_K \left( \frac{i}{n} \right) \right) + O_K \left( \frac{i}{n} \right) \cdot \left( 1 + O_K \left( \frac{i}{n} \right) \right)\\
\nonumber &= -  \sin \frac{t_{ij}}{\sqrt{n}} \cdot \left(1 + O_K \left( \frac{N }{n} \right)  \right) + O_K \left( \frac{1}{n} \right) \cdot \left( N + O_K \left( \frac{N^2}{n} \right) \right)  \\
&= - \sin \frac{t_{ij}}{\sqrt{n}} + O_K \left( \frac{N}{n} \right).
\end{align}
Now, by the definition of $\alpha_{ij}$ and the identity $\arccos (x) = \tfrac{\pi}{2} + \arcsin(-x)$, we have
\begin{align*}
\alpha_{ij} &= \arccos \left(  - \sin \frac{t_{ij}}{\sqrt{n}} + O_K \left( \frac{N}{n} \right) \right) = \arcsin \left(  \sin \frac{t_{ij}}{\sqrt{n}} + O_K \left( \frac{N}{n} \right) \right) + \frac{\pi}{2}.
\end{align*}
Therefore, by the definition of $\tilde{\alpha}_{ij}$, we have proved
\begin{align*}
\tilde{\alpha}_{ij} &= \sqrt{n} \cdot \arcsin \left( \sin \frac{t_{ij}}{\sqrt{n}} + O_K \left(\frac{N}{n} \right) \right) = \sqrt{n} \cdot \arcsin \left( \frac{t_{ij}}{\sqrt{n}} + O_K \left(\frac{N}{n} \right) \right) = t_{ij} + O_K \left( \frac{N}{\sqrt{n}} \right),
\end{align*}
where we used the estimate $\arcsin (x) = x + O(x^3)$. \par 
It remains to prove that the difference $t_{ij} - \tilde{\alpha}_{ij}$ only depends on the variables $\tilde{\alpha}_{12}, \ldots,$ $\tilde{\alpha}_{N-1,N}$. To see this, we use induction on $i$. For the base case $i = 1$, we have $t_{ij} = \tilde{\alpha}_{ij}$ so that $g_{ij}$ is identically $0$. Let us therefore assume that $i \geqslant 2$. It follows from (\ref{p1-thedotproductfromhell}) that there are functions $A \lp t_{12}, t_{13}, \ldots, t_{i-1, j} \rp$ and $B \lp t_{12}, t_{13}, \ldots, t_{i-1,j} \rp \neq 0$, both depending on $n$, such that
\begin{align}\label{p1-AnBn}
-\sin \lp \frac{t_{ij}}{\sqrt{n}} \rp = \frac{\textbf{u}_i \cdot \textbf{u}_j - A \lp t_{12}, t_{13}, \ldots, t_{i-1, j} \rp}{B \lp t_{12}, t_{13}, \ldots, t_{i-1,j} \rp}.
\end{align}
By induction, each of the variables $t_{ab}$ ($1 \leqslant a \leqslant i-1$, $1 \leqslant b \leqslant j$, $a<b$) can be written as a function of the variables $\tilde{\alpha}_{cd}$ ($1 \leqslant c < d \leqslant N)$. Moreover, by definition of $\alpha_{ij}$, we have 
\begin{align*}
\textbf{u}_i \cdot \textbf{u}_j = \cos \lp \frac{\tilde{\alpha}_{ij}}{\sqrt{n}} + \frac{\pi}{2} \rp.
\end{align*}
Thus we see from (\ref{p1-AnBn}) that for some function $C$,
\begin{align*}
t_{ij} = C \lp \tilde{\alpha}_{12}, \tilde{\alpha}_{13}, \ldots, \tilde{\alpha}_{N-1,N} \rp.
\end{align*}
Finally, we take $g_{ij} = C - \tilde{\alpha}_{ij}$.
\end{proof}
\noindent \textsc{Remarks.}\par 
1) The proof of the lemma also shows that, if we have $\left| t_{ij} \right| > K+1$ for some admissible $(i,j)$, then for $n$ large enough, we even get $\left| \tilde{\alpha}_{ij} \right| > K$, cf. \cite[Sect. 3]{soda10}. 
\par 2) Let $I(\tilde{\alpha}_{ij})$ be short for $\smash{I \bigl( c_{ij} < \tilde{\alpha}_{ij} \leqslant c_{ij}'\bigr) }$, and let $I(t_{ij})$ be short for $\smash{I \bigl( c_{ij} < t_{ij} \leqslant c_{ij}'\bigr) }$. Then, if $t_{ij} \in \lb -K-1, K+1 \rb$ and $I(\tilde{\alpha}_{ij}) - I(t_{ij}) \neq 0$, we see that $t_{ij}$ belongs to a set $U_{ij}$ with $|U_{ij}| \leqslant C_K N/\sqrt{n} $ for some constant $C_K$ independent of $n$ and $N$. From this it follows that for $t_{ij} \in \lb -K-1, K+1\rb$, the difference of indicator functions $I(\tilde{\alpha}_{ij}) - I(t_{ij})$ is dominated in absolute value by the function
\begin{align*}
d_{ij}(t_{ij}) := 
\begin{cases}
1 \quad &\text{if $t_{ij} \in U_{ij}$,}\\
0 \quad &\text{otherwise.}
\end{cases}
\end{align*}
\par    
We are now ready to give a proof of Theorem \ref{p1-pointsonspheretheorem}.\\\\
\textit{Proof of Theorem \ref{p1-pointsonspheretheorem}.} 
We note for future use that, because of the conditions on the sequences $\lbrace c_{ij} \rbrace$ and $\lbrace c_{ij}' \rbrace$, we have
\begin{align}\label{p1-xi}
\int_{c_{ij}}^{c_{ij}'} e^{-t^2 /2} \hspace{0.1cm} dt \geqslant \int_{K-\varepsilon}^K e^{-t^2 / 2} \hspace{0.1cm} dt =: \xi(K, \varepsilon) > 0
\end{align}
for all admissible $(i,j)$. \par Let us write $d\textbf{t} = dt_{(N-1)N} \cdots dt_{12}$. We are studying the probability 
\begin{align*}
\mathbb{P} \bigl( c_{ij} < \tilde{\alpha}_{ij} \leqslant c_{ij}' : 1 \leqslant i < j \leqslant N \bigr),
\end{align*} 
which equals
\begin{align*}
P(N) = \left( \prod_{l = 1}^{N-1} \prod_{m=1}^l \frac{\omega_{n-m}}{\omega_{n-m+1} \sqrt{n}} \right) \int_{-\sqrt{n}\frac{\pi}{2}}^{\sqrt{n}\frac{\pi}{2}} \cdots \int_{-\sqrt{n}\frac{\pi}{2}}^{\sqrt{n}\frac{\pi}{2}}  \prod_{i,j = 1 \atop i < j}^N I(\tilde{\alpha}_{ij}) \cos \left( \frac{t_{ij}}{\sqrt{n}} \right)^{n-i-1} \hspace{0.1cm} d\textbf{t},
\end{align*}
cf. \cite[eq. (3.3)]{soda10}. Since the integrand vanishes if $\left| \tilde{\alpha}_{ij} \right| > K$ for some $(i,j)$, the first remark after Lemma \ref{p1-lemma2.4} implies that for $n$ large, we may restrict the domain of integration to the box $\smash{\lb - K - 1, K + 1 \rb^{\nCr{N}{2}}}$. Then, by Lemma \ref{p1-lemma2.2} and Lemma \ref{p1-lemma2.3},
\begin{align*}
P(N) &= \left( \frac{1}{\sqrt{2 \pi}} \right)^{{{N}\choose{2}}} \cdot \left( 1 + O_K \left( \frac{N^3}{n} \right) \right) \int_{-K-1}^{K+1} \cdots \int_{-K-1}^{K+1} \prod_{i,j = 1 \atop i < j}^N I(\tilde{\alpha}_{ij}) \cdot e^{-t_{ij}^2 / 2} \hspace{0.1cm} d \textbf{t}.
\end{align*}
\par 
Let $b(N) = \smash{{{N}\choose{2}}}$. We are now interested in obtaining an asymptotic expression for the quotient $P(N)/G(N)$. We have
\begin{align}\label{p1-step0}
\frac{P(N)}{G(N)} = \left( 1 + O_K \left( \frac{N^3}{n} \right) \right) Q(N)
\end{align}
with 
\begin{align*}
Q(N) = \int_{-K-1}^{K+1} \cdots \int_{-K-1}^{K+1} \prod_{i,j = 1 \atop i < j}^N I(\tilde{\alpha}_{ij}) \cdot e^{-t_{ij}^2 / 2} \hspace{0.1cm} d \textbf{t} \cdot \left( \int_{-K}^K \cdots \int_{-K}^K \prod_{i,j = 1 \atop i < j}^N I(t_{ij}) e^{-t_{ij}^2 / 2} \hspace{0.1cm} d \textbf{t} \right)^{-1}.
\end{align*}
\noindent We want to obtain upper and lower bounds on the quantity $Q(N)$ by using the bounds on $I(\tilde{\alpha}_{ij})$ given in the remarks following the proof of Lemma \ref{p1-lemma2.4}, namely
\begin{align*}
I(t_{ij}) - d_{ij}(t_{ij}) \leqslant I(\tilde{\alpha}_{ij}) \leqslant I(t_{ij}) + d_{ij}(t_{ij})
\end{align*}
for all admissible $(i,j)$. To this end, let $\mathcal{B}(N) = \big\{ (i,j) : 1 \leqslant i < j \leqslant N \big\}$. Then, by using (\ref{p1-xi}) and the definition of $d_{ij}$, we find that
\begin{align}\label{p1-Qlargerthan-new}
\nonumber Q(N) 
&\geqslant \prod_{i, j = 1 \atop i < j}^N \int_{-K-1}^{K+1} \left( I(t_{ij}) - d_{ij}(t_{ij}) \right) e^{-t_{ij}^2 / 2} \hsp dt_{ij} \cdot \left( \int_{-K}^K I(t_{ij})e^{-t_{ij}^2 / 2} \hsp dt_{ij} \right)^{-1}\\
\nonumber 
&= \prod_{i, j = 1 \atop i < j}^N \left( 1 - \int_{-K-1}^{K+1} d_{ij}(t_{ij})e^{-t_{ij}^2 / 2} \hsp d t_{ij} \cdot \left( \int_{-K}^K I(t_{ij})e^{-t_{ij}^2 / 2} \hsp dt_{ij} \right)^{-1} \right)
\\ &\geqslant \prod_{i, j = 1 \atop i < j}^N \left( 1 - \frac{4 C_K N}{\xi (K, \varepsilon) \sqrt{n}} \right) = 1 + \sum_{D \subset \mathcal{B}(N) \atop D \neq \varnothing} (-1)^{\# D} \left(\frac{4 C_K N}{\xi (K, \varepsilon) \sqrt{n} } \right)^{\# D}.
\end{align}
Similarly, we have
\begin{align}\label{p1-Qsmallerthan-new}
\nonumber Q(N) &\leqslant \prod_{i, j = 1 \atop i < j}^N \int_{-K-1}^{K+1} \left( I(t_{ij}) + d_{ij}(t_{ij}) \right) e^{-t_{ij}^2 / 2} \hsp dt_{ij} \cdot \left( \int_{-K}^K I(t_{ij})e^{-t_{ij}^2 / 2} \hsp dt_{ij} \right)^{-1} \\
\nonumber 
&= \prod_{i, j = 1 \atop i < j}^N \left( 1 + \int_{-K-1}^{K+1} d_{ij}(t_{ij}) e^{-t_{ij}^2 / 2} \hsp dt_{ij}  \cdot \left( \int_{-K}^K I(t_{ij})e^{-t_{ij}^2 / 2} \hsp dt_{ij} \right)^{-1} \right) 
\\ &\leqslant \prod_{i, j = 1 \atop i < j}^N \left( 1 + \frac{4 C_K N}{\xi (K, \varepsilon) \sqrt{n}} \right) = 1 + \sum_{D \subset \mathcal{B}(N) \atop D \neq \varnothing} \left( \frac{4 C_K N}{\xi (K, \varepsilon) \sqrt{n} } \right)^{\# D}.
\end{align}
If we let
\begin{align*}
R(N) := \sum_{D \subset \mathcal{B}(N) \atop D \neq \varnothing} \left( \frac{4 C_K N}{\xi (K, \varepsilon) \sqrt{n} } \right)^{\# D},
\end{align*}
then (\ref{p1-Qlargerthan-new}) and (\ref{p1-Qsmallerthan-new}) imply that for $n \gg 1$,
\begin{align}\label{p1-jakobssons}
1 - R(N) \leqslant Q(N) \leqslant 1 + R(N).
\end{align}
We now show that $R(N) \longrightarrow 0$ as $n \longrightarrow \infty$. Defining
\begin{align*}
S(N) = \frac{4 C_K N}{\xi (K, \varepsilon) \sqrt{n}},
\end{align*}
we find that
\begin{align}\label{p1-Rbound}
R(N) = \sum_{m = 1}^{b(N)} {{b(N)}\choose{m}} S(N)^m = -1 + \sum_{m = 0}^{b(N)} {{b(N)}\choose{m}} S(N)^m = -1 + (1 + S(N))^{b(N)}.
\end{align}
Therefore we need to show that $(1 + S(N))^{b(N)} \longrightarrow 1$ as $n \longrightarrow \infty$. However, since $N = o \lp n^{1/6} \rp$ and $b(N) = O \lp N^2 \rp$, this follows from the estimate
\begin{align*}
(1 + S(N))^{b(N)} &= \exp \big( b(N) \cdot \log (1 + S(N)) \big) \\ &= \exp \Bigl( b(N) \cdot \lp S(N) + O\lp S(N)^2 \rp \rp \Bigr) \\ &= 1 + O\big( b(N) S(N) \big).
\end{align*}
Thus we obtain from (\ref{p1-Rbound}) that
\begin{align*}
R(N) = O\big(b(N)S(N)\big) = O_{K , \varepsilon} \left( \frac{N^3}{\sqrt{n}} \right).
\end{align*}
This completes the proof of the theorem. $\hfill \blacksquare$ \\ 
\par We end this section by proving a lemma that will be needed in the proof of Theorem \ref{p1-maintechnical}.
\begin{lemma}\label{p1-lemma2.5}
\textit{Let $N = o \lp n^{1/6} \rp$ be an integer. For all admissible $(i,j)$, let $\phi_{ij}$ and $\tilde{\alpha}_{ij}$ be given as in Lemma \ref{p1-lemma2.4}. Suppose that $K > 0$ and that $\left|\tilde{\alpha}_{ij} \right| \leqslant K$ for all admissible $(i,j)$. Then we have}
\begin{align*}
\prod_{i,j = 1 \atop i < j}^N \sin \lp \phi_{ij} \rp^{n-N-1} \cos \left( \frac{\tilde{\alpha}_{ij}}{\sqrt{n}} \right) = \bigl( 1 + h \lp \tilde{\alpha}_{12}, \ldots, \tilde{\alpha}_{N-1,N} \rp \bigr) \prod_{i,j = 1 \atop i < j}^N e^{-\tilde{\alpha}_{ij}^2 / 2}
\end{align*}
for some function $h \lp \tilde{\alpha}_{12}, \ldots, \tilde{\alpha}_{N-1,N} \rp$, which depends on $n$ and $N$ and satisfies $h = O_K \lp N^3 / \sqrt{n} \rp$.
\end{lemma}
\begin{proof}
We prove the bound 
\begin{align*}
\prod_{i,j = 1 \atop i < j}^N \lp \sin (\phi_{ij})^{n-N-1} \cos \left( \frac{\tilde{\alpha}_{ij}}{\sqrt{n}} \right)   e^{\tilde{\alpha}_{ij}^2 / 2} \rp = 1 + O_K \lp \frac{N^3}{\sqrt{n}} \rp.
\end{align*}
From this it follows immediately from Lemma \ref{p1-lemma2.4} and the first remark after Lemma \ref{p1-lemma2.4} that with $t_{ij} = \sqrt{n}(\phi_{ij} - \frac{\pi}{2})$, we can take 
\begin{align*}
h \lp \tilde{\alpha}_{12}, \ldots, \tilde{\alpha}_{N-1,N} \rp := -1 + \prod_{i,j = 1 \atop i < j}^N \cos \lp \frac{t_{ij}}{\sqrt{n}} \rp^{n-N-1} \cos \lp \frac{\tilde{\alpha}_{ij}}{\sqrt{n}} \rp e^{\tilde{\alpha}_{ij}^2 / 2}.
\end{align*}
\par From the definition of $t_{ij}$, we see that
\begin{align*}
\prod_{i,j = 1 \atop i < j}^N \sin (\phi_{ij})^{n-N-1} = \prod_{i,j = 1 \atop i < j}^N \cos \left( \frac{t_{ij}}{\sqrt{n}} \right)^{n-N-1}.
\end{align*}
Furthermore, by Taylor expansions and Lemma \ref{p1-lemma2.4}, we see that
\begin{align*}
\log \left( \prod_{i,j = 1 \atop i < j}^N \cos \left( \frac{t_{ij}}{\sqrt{n}} \right)^{n-N-1} \right) &= (n-N-1) \sum_{i,j = 1 \atop i < j}^N \log \left( 1 - \frac{t_{ij}^2}{2n} + O_K \left( \frac{1}{n^2} \right) \right) \\
&= (n-N-1) \sum_{i,j = 1 \atop i < j}^N \left( - \frac{t_{ij}^2}{2n} + O_K \left( \frac{1}{n^2} \right) \right) \\
&= (n-N-1) \sum_{i,j = 1 \atop i < j}^N \left( - \frac{\tilde{\alpha}_{ij}^2}{2n} + O_K \left( \frac{N}{n^{3/2}} \right) \right) \\
&= - \hspm \sum_{i,j = 1 \atop i < j}^N \frac{\tilde{\alpha}_{ij}^2}{2} + O_K \left( \frac{N^3}{\sqrt{n}} \right).
\end{align*}
Therefore,
\begin{align}\label{p1-thm4.1-7}
\prod_{i,j = 1 \atop i < j}^N \sin (\phi_{ij})^{n-N-1} &= \exp \left(- \hspm \sum_{i,j = 1 \atop i < j}^N \frac{\tilde{\alpha}_{ij}^2}{2} + O_K \left( \frac{N^3}{\sqrt{n}} \right) \right) = \left( 1+ O_K \left( \frac{N^3}{\sqrt{n}} \right) \right) \prod_{i,j = 1 \atop i < j}^N e^{-\tilde{\alpha}_{ij}^2 / 2}.
\end{align}
On the other hand,
\begin{align*}
\prod_{i,j = 1 \atop i < j}^N \cos \left( \frac{\tilde{\alpha}_{ij}}{\sqrt{n}} \right) &= \exp \left( \sum_{i,j = 1 \atop i < j}^N \log \left( \cos \left( \frac{\tilde{\alpha}_{ij}}{\sqrt{n}} \right) \right) \right) = \exp \left( - \hspm \sum_{i,j = 1 \atop i < j}^N \lp \frac{\tilde{\alpha}_{ij}^2}{2n} + O_K \left( \frac{1}{n^2} \right) \rp \right) \\
&= \exp \left( O_K \left( \frac{N^2}{n} \right) + O_K \left( \frac{N^2}{n^2} \right) \right) = 1 + O_K \left( \frac{N^2}{n} \right).
\end{align*}
Together with (\ref{p1-thm4.1-7}), this proves the lemma. 
\end{proof}
\section{The Main Technical Result}
We now state the main technical result needed to prove Theorem \ref{p1-maintheorem}. In its proof we follow the arguments in \cite{soda10}. \par 
In order to state the result, we first introduce some notation. First of all, we will continue to denote by $b(N)$ the binomial coefficient $\nCr{N}{2}$. Furthermore, if $\lambda = N + \ell$ where $\ell \geqslant 0$ is an integer, and $\smash{f_N : \lp \R_{\geqslant 0} \rp^{b(N)+\lambda} \longrightarrow \R_{\geqslant 0}}$ is a function, we will write $f_N = f_{0,N} \otimes f_{1,N}$ if there are functions
\begin{align*}
f_{0,N} : \lp \R_{\geqslant 0} \rp^{\lambda} \longrightarrow \R_{\geqslant 0}, \quad \quad f_{1,N} : \lp \R_{\geqslant 0} \rp^{b(N)} \longrightarrow \R_{\geqslant 0} 
\end{align*}
with $\int f_{1,N} > 0$, such that for any non-negative real numbers $x_1, \ldots, x_{\lambda}$ and $y_1, \ldots, y_{b(N)}$, $f_N$ has the factorization property
\begin{align*}
f_N \lp x_1, \ldots, x_\lambda, y_1, \ldots, y_{b(N)} \rp = f_{0,N} \lp x_1, \ldots, x_\lambda \rp f_{1,N} \lp y_1, \ldots, y_{b(N)} \rp.
\end{align*}
To ease the notation, we will omit the $N$ subscript from the functions $f_{0,N}, \, f_{1,N}$, and $f_N$ and simply denote these functions by $f_0, \, f_1$, and $f$. We also recall that $M_\lambda$ denotes the set $\lbr \textbf{n} = \lp n_1, \ldots, n_\lambda \rp \in \Z_{+}^\lambda : n_i = n_j \Leftrightarrow i = j \rbr$.
\begin{theorem}\label{p1-maintechnical}
\textit{Let $N = N(n) \geqslant 2$ be an integer with $N = o \lp n^{1/6} \rp$. Let $\lambda := N+\ell$ with $\ell \geqslant 0$ an integer. Let $K, M > 0$ be real numbers and consider the family}
\begin{align*}
F\lp K, M, N, \ell \rp = \left \lbrace {\small \begin{array}{l}
    \text{$f : (\R_{\geqslant 0})^{b(N)+\lambda} \longrightarrow \R_{\geqslant 0}$ is Borel-measurable, }  \\
     f = f_0 \otimes f_1, \quad \big\| f_0 \big\|_\infty, \, \big\| f_1 \big\|_\infty  \leqslant \sqrt{M}, \quad \text{supp}(f) \subset \lb 0,K \rb^{\nCr{N}{2} + \lambda}
  \end{array}}
\right \rbrace. 
\end{align*}
Then, for any $f \in F \lp K, M, N, \ell \rp$, there exists a function $C\lp n, N, K, f_1 \rp$ such that
\begin{align*}
&\hspace{-0.5cm} \expect \lb \sum_{\textbf{n} \in M_\lambda} f\big(\mathcal{V}_{n_1}, \ldots, \mathcal{V}_{n_\lambda}, \tilde{\varphi}_{n_1 n_2}, \ldots, \tilde{\varphi}_{n_{N-1} n_N} \big) \rb 
\\ &= \bigl( 1 + C(n, N, K, f_1) \bigr) \expect \lb \sum_{\textbf{n} \in M_\lambda} f\left(T_{n_1}, \ldots, T_{n_\lambda}, \Phi_{n_1 n_2}, \ldots, \Phi_{n_{N-1} n_N} \right) \rb
\\ &\quad \quad + O \left( 2^{- \lambda} 5^{\lfloor \lambda^2 / 4 \rfloor} (\sqrt{3}/2)^n M (K+1)^\lambda \right)
\end{align*}
and $C \lp n, N, K, f_1 \rp = O_K \lp N^3 / \sqrt{n} \rp$, where the implied constant is independent of $f$.
\end{theorem}
\begin{proof}
Let $f \in F\lp K, M, N, \ell \rp$. For $\textbf{x}, \textbf{y} \in \R^n$, let us write $\tilde{\varphi} (\textbf{x}, \textbf{y}) = \sqrt{n} \lp \tfrac{\pi}{2} - \varphi (\textbf{x}, \textbf{y} ) \rp$ where $\varphi (\textbf{x}, \textbf{y}) \in \lb 0, \tfrac{\pi}{2} \rb$ is the angle between $\pm \textbf{x}$ and $\pm \textbf{y}$. Let $V_n$ denote the volume of the $n$-dimensional unit ball and define the function $\tilde{f} : \lp \R^n \rp^\lambda \rightarrow \R$ by
\begin{align*}
\tilde{f}(\textbf{x}_1, \ldots, \textbf{x}_\lambda) =  f\left( V_n |\textbf{x}_1|^n, \ldots, V_n |\textbf{x}_\lambda |^n, \tilde{\varphi}(\textbf{x}_1, \textbf{x}_2), \ldots, \tilde{\varphi}(\textbf{x}_{N-1}, \textbf{x}_N )\right)
\end{align*}
if $\textbf{x}_1, \ldots, \textbf{x}_N \neq 0$, or $\tilde{f}\lp \textbf{x}_1, \ldots, \textbf{x}_\lambda \rp = 0$ otherwise. Proceeding as in \cite{soda10}, using the mean value formula due to Rogers \cite[Thm. 4]{rogers55}, we find that
\begin{align}\label{p1-thm4.1-1}
\nonumber &\hspace{-0.5cm} \expect \lb \sum_{\textbf{n} \in M_\lambda} f\bigl( \mathcal{V}_{n_1}, \ldots, \mathcal{V}_{n_\lambda}, \tilde{\varphi}_{n_1 n_2}, \ldots, \tilde{\varphi}_{n_{N-1} n_N} \bigr) \rb \\ \nonumber 
&= 2^{-\lambda} \int_{\R^n} \cdots \int_{\R^n} \tilde{f}(\textbf{x}_1, \ldots, \textbf{x}_\lambda) I(\textbf{x}_i = \pm \textbf{x}_j \Leftrightarrow i = j) \hspace{0.1cm} d\textbf{x}_1 \cdots d\textbf{x}_\lambda
\\ \nonumber & \quad + 2^{-\lambda} \sum_{(\nu, \mu)} \sum_{q \geqslant 1} \sum_{D} \left( \frac{e_1}{q} \cdots \frac{e_m}{q} \right)^n \int_{\R^n} \cdots \int_{\R^n} \tilde{f} \left( \sum_{h = 1}^m \frac{d_{h1}}{q}\textbf{x}_h, \ldots, \sum_{h = 1}^m \frac{d_{h \lambda}}{q}\textbf{x}_h \right) 
\\ &\quad \quad \times I\left( \sum_{h = 1}^m \frac{d_{hi}}{q}\textbf{x}_h = \pm \sum_{h = 1}^m \frac{d_{hj}}{q}\textbf{x}_h \Leftrightarrow i = j\right) \hspace{0.1cm} d\textbf{x}_1 \cdots d\textbf{x}_m.
\end{align}
Since $D$ has dimensions $m \times k$, $m < k$, we immediately see that the last indicator function above vanishes if all entries of the matrix $D$ belong to the set $\lbrace 0, \pm 1 \rbrace$ and every column of $D$ has exactly one non-zero entry. Hence we can restrict the innermost summation above to take place over those $D$ satisfying one of the following (mutually exclusive) conditions:
\begin{itemize}
\item[1)] All entries of $D$ belong to $\lbrace 0, \pm 1 \rbrace$, and $D$ has at least one column with at least two non-zero entries.
\item[2)] At least one entry of $D$ is, in absolute value, at least $2$.
\end{itemize}
Let $B_K$ be the closed $n$-dimensional ball of volume $K$ centered at the origin. By \cite[Lemma 7]{rogers56} and \cite[Lemma 8]{rogers56}, restricting the summation to those $D$ satisfying 1) or 2) above gives us the estimate
\begin{align}\label{p1-thm4.1-2}
\nonumber 0 &\leqslant \sum_{(\nu,\mu)} \sum_{q \geqslant 1} \sum_D \left( \frac{e_1}{q} \cdots \frac{e_m}{q} \right)^n \int_{\R^n} \cdots \int_{\R^n} \prod_{j = 1}^\lambda I_{B_K} \left( \sum_{i = 1}^m \frac{d_{ij}}{q}\textbf{x}_i \right) \hsp d\textbf{x}_1 \cdots d\textbf{x}_m \\
\nonumber &\leqslant 2 \cdot 3^{\lfloor \lambda^2 / 4 \rfloor} (3/4)^{n/2} (K+1)^\lambda + 21 \cdot 5^{\lfloor \lambda^2 / 4 \rfloor} 2^{-n} (K+1)^\lambda \\
&\ll 5^{\lfloor \lambda^2/4 \rfloor} (\sqrt{3}/2)^n (K+1)^\lambda.
\end{align}
Now, since $|f| \leqslant M$ and $\text{supp}(f) \subset \lb 0,K \rb^{\lambda + b(N)}$, we find that $|\tilde{f}| \leqslant M$ and $\text{supp}\lp \tilde{f} \rp \subset B_K^\lambda$. Therefore we have
\begin{align}\label{p1-thm4.1-3}
\nonumber &\left| \frac{1}{2^\lambda} \sum_{(\nu,\mu)} \sum_{q \geqslant 1} \sum_D \left( \frac{e_1}{q} \cdots \frac{e_m}{q} \right)^n \int_{\R^n} \cdots \int_{\R^n} \tilde{f} \left( \sum_{h = 1}^m \frac{d_{h1}}{q}\textbf{x}_h, \ldots, \sum_{h = 1}^m \frac{d_{h \lambda}}{q} \textbf{x}_h \right) \right.  \\ 
\nonumber &\quad \quad \left.  \times I \left( \sum_{h = 1}^m \frac{d_{hi}}{q} \textbf{x}_h  = \pm \sum_{h = 1}^m \frac{d_{hj}}{q} \textbf{x}_h \Leftrightarrow i = j \right) \hsp d\textbf{x}_1 \cdots d\textbf{x}_m  \right|\\
&\leqslant \frac{M}{2^\lambda}  \sum_{(\nu,\mu)} \sum_{q \geqslant 1} \sum_D \left( \frac{e_1}{q} \cdots \frac{e_m}{q} \right)^n \int_{\R^n} \cdots \int_{\R^n} \prod_{j = 1}^\lambda I_{B_K} \left( \sum_{h = 1}^m \frac{d_{hj}}{q}\textbf{x}_h \right) \hsp d\textbf{x}_1 \cdots d\textbf{x}_m.
\end{align}
Along with (\ref{p1-thm4.1-1}) and (\ref{p1-thm4.1-2}) this proves the estimate
\begin{align}\label{p1-thm4.1-4}
\nonumber &\expect \lb \sum_{\textbf{n} \in M_\lambda} f\left(\mathcal{V}_{n_1}, \ldots, \mathcal{V}_{n_\lambda}, \tilde{\varphi}_{n_1 n_2}, \ldots, \tilde{\varphi}_{n_{N-1} n_N} \right) \rb \\ &\quad \quad = \frac{1}{2^\lambda} \int_{\R^n} \cdots \int_{\R^n} \tilde{f}(\textbf{x}_1, \ldots, \textbf{x}_\lambda) \hsp d\textbf{x}_1 \cdots d\textbf{x}_\lambda + O\left(2^{-\lambda} 5^{\lfloor \lambda^2/4 \rfloor} \lp \sqrt{3}/2 \rp^n M (K+1)^\lambda\right).
\end{align}
It remains to understand the integral. Following the argument in \cite{soda10}, by changing to spherical coordinates and using the symmetry of $\tilde{f}$, we see that
\begin{align}\label{p1-sphintegral}
\nonumber &\hspace{-0.5cm}\int_{\R^n} \cdots \int_{\R^n} \tilde{f}(\textbf{x}_1, \ldots, \textbf{x}_\lambda) \hsp d\textbf{x}_1 \cdots d\textbf{x}_\lambda 
\\ \nonumber &= \omega_n^{\lambda - N}  \left( \prod_{h = 0}^{N-1} \omega_{n-h} \right) \int_0^\infty \cdots \int_0^\infty \int_0^\pi \cdots \int_0^\pi f\left(V_n r_1^n, \ldots, V_n r_\lambda^n, \tilde{\varphi}\lp \textbf{u}_1, \textbf{u}_2 \rp, \ldots, \tilde{\varphi} \lp \textbf{u}_{N-1}, \textbf{u}_N \rp \right) 
\\ &\quad \quad \times \left( \prod_{j = 1}^\lambda r_j^{n-1} \right) \prod_{i,j = 1 \atop i < j}^N \sin (\phi_{ij})^{n-i-1} \hsp d\phi_{N-1,N} \cdots d\phi_{12} d r_\lambda \cdots dr_1,
\end{align}
where the $\textbf{u}_i$ are unit vectors given in Lemma \ref{p1-lemma2.4}. \par 
We note that for any admissible $(i,j)$, one has $\varphi \lp \textbf{u}_i, \textbf{u}_j \rp = \tfrac{\pi}{2} - \left| \tfrac{\pi}{2} - \arccos \lp \textbf{u}_i \cdot \textbf{u}_j \rp \right|$. For this reason, we want to change variables in the above integral and integrate with respect to 
\begin{align*}
\alpha_{ij} := \arccos\lp \textbf{u}_i \cdot \textbf{u}_j\rp, \quad \quad 1 \leqslant i < j \leqslant N.
\end{align*}
Therefore, consider the map $J_N \lp \phi_{12}, \ldots, \phi_{N-1,N} \rp = \lp \alpha_{12}, \ldots, \alpha_{N-1,N}\rp$ with
\begin{align}\label{p1-fuckedupfunction}
\alpha_{ij} = 
\begin{cases}
\phi_{ij}, &\text{if } i = 1,\\
\arccos \left(F\left(\phi_{1i}, \phi_{1j}, \ldots, \phi_{i-1,j} \right) + \cos (\phi_{ij}) X\left(\phi_{1i}, \phi_{1j}, \ldots, \phi_{i-1,j} \right) \right), &\text{if } i > 1,
\end{cases}
\end{align}
and
\begin{align*}
F(\phi_{1i}, \phi_{1j}, \ldots, \phi_{i-1,j}) &= \sum_{m = 1}^{i-1} \cos \phi_{mi}\cos \phi_{mj} \prod_{\ell = 1}^{m-1} \sin \phi_{ \ell i} \sin \phi_{\ell j}, \\
X(\phi_{1i}, \phi_{1j}, \ldots, \phi_{i-1,j}) &= \prod_{m = 1}^{i-1} \sin \phi_{mi} \sin \phi_{mj},
\end{align*}
cf. \cite[Eq. (3.8)]{soda10}. As in \cite{soda10}, we see that $J_N$ is a diffeomorphism from $(0,\pi)^{b(N)}$ to 
\begin{align*}
\Omega_N := J_N \left( (0,\pi)^{b(N)} \right)
\end{align*}
with Jacobian determinant 
\begin{align*}
\textbf{J}(J_N) = \prod_{i, j = 1 \atop i < j}^N \frac{\sin(\phi_{ij})^{N-i}}{\sin \alpha_{ij}}.
\end{align*}
Hence, by changing coordinates from $(r_1, \ldots, r_\lambda, \phi_{12}, \ldots, \phi_{N-1,N})$ to 
\begin{align*}
\left(\lp s_1 V_n^{-1} \rp^{1/n}, \ldots, \lp s_\lambda V_n^{-1}\rp^{1/n}, \alpha_{12}, \ldots, \alpha_{N-1,N}\right),
\end{align*}
we can write the integral (\ref{p1-sphintegral}) as
\begin{align}\label{p1-nextintegral}
\nonumber &\left( \prod_{h = 1}^{N-1} \prod_{m = 1}^h \frac{\omega_{n-m}}{\omega_{n-m+1}} \right) \int_0^\infty \cdots \int_0^\infty \int_{\Omega_N} f\left( s_1, \ldots, s_\lambda, \sqrt{n} \left| \alpha_{12} - \tfrac{\pi}{2} \right|, \ldots, \sqrt{n} \left| \alpha_{N-1,N} - \tfrac{\pi}{2} \right| \right) 
\\ &\quad \quad \times \prod_{i,j = 1 \atop i < j}^{N} \left( \sin(\phi_{ij})^{n-N-1} \sin \alpha_{ij} \right) \hsp d \alpha_{N-1,N} \cdots d \alpha_{12} ds_\lambda \cdots ds_1.
\end{align}
Changing coordinates again, we want to let $\alpha_{ij} = n^{-1/2} \tilde{\alpha}_{ij} + \tfrac{\pi}{2}$ for all admissible $(i,j)$, so that
\begin{align*}
\lp \alpha_{12}, \ldots, \alpha_{N-1,N} \rp = n^{-\frac{1}{2}} \lp \tilde{\alpha}_{12}, \ldots, \tilde{\alpha}_{N-1,N} \rp + \textbf{p}, \quad \quad \textbf{p} = \left( \tfrac{\pi}{2}, \ldots, \tfrac{\pi}{2} \right).
\end{align*}
Since $f$ vanishes if one of its arguments lies outside of $\lb 0, K \rb$, this change of variables will make sense provided that
\begin{align}\label{p1-theultimatecontainment}
n^{-\frac{1}{2}} \lb -K, K \rb^{b(N)} + \textbf{p} \subset \Omega_N.
\end{align}
At a first glance, it is not clear whether this inclusion holds as $n$ tends to infinity, given that $N$ depends on $n$. (For example, it is conceivable that, with $N = o \lp n^{1/6} \rp$, the diameter of $\Omega_N$ decreases at a rate faster than $n^{-1/2}$.) Since the proof of (\ref{p1-theultimatecontainment}) is somewhat involved, we establish the inclusion for $n$ large enough in the appendix. Assuming (\ref{p1-theultimatecontainment}) for now, the change of variables, (\ref{p1-sphintegral}), (\ref{p1-nextintegral}), Lemma \ref{p1-lemma2.3}, and Lemma \ref{p1-lemma2.5} show that
\begin{align}\label{p1-thm4.1-5}
\nonumber &\hspace{-0.5cm}\int_{\R^n} \cdots \int_{\R^n} \tilde{f}(\textbf{x}_1, \ldots, \textbf{x}_\lambda) \hsp d\textbf{x}_1 \cdots d\textbf{x}_\lambda 
\\[+0.5em] \nonumber &= \left( \prod_{h = 1}^{N-1} \prod_{m=1}^{h} \frac{\omega_{n-m}}{\omega_{n-m+1} \sqrt{n}} \right) \int_0^\infty \cdots \int_0^\infty \int_{-\infty}^\infty \cdots \int_{-\infty}^\infty f(s_1, \ldots, s_\lambda, |\tilde{\alpha}_{12}|, \ldots, |\tilde{\alpha}_{N-1,N} |) 
\\ \nonumber &\quad \quad \times \prod_{i,j = 1 \atop i < j}^N \sin (\varphi_{ij})^{n-N-1} \cos \left( \frac{\tilde{\alpha}_{ij}}{\sqrt{n}} \right) \hsp d \textbf{a} d\textbf{s}
\\ \nonumber &= \lp \frac{1}{\sqrt{2 \pi}} \rp^{\nCr{N}{2}} \bigl( 1 + p \lp n, N \rp \bigr) \int_0^\infty \cdots \int_0^\infty \int_{-\infty}^{\infty} \cdots \int_{-\infty}^{\infty} f_0 \lp s_1, \ldots, s_\lambda \rp f_1 \lp \left| \tilde{\alpha}_{12} \right|, \ldots, \left| \tilde{\alpha}_{N-1,N} \right| \rp \\ 
&\quad \quad \times \lp 1 + h \lp \tilde{\alpha}_{12}, \ldots, \tilde{\alpha}_{N-1,N} \rp  \rp \prod_{i,j = 1 \atop i < j}^N \exp \lp -\tilde{\alpha}_{ij}^2 / 2 \rp \hsp d \textbf{a} d \textbf{s},
\end{align}
where we wrote $d \textbf{a} = d\tilde{\alpha}_{N-1,N} \cdots d\tilde{\alpha}_{12}$ and $d \textbf{s} = ds_\lambda \cdots d s_1$. Continuing, we see that this integral equals
\begin{align}\label{p1-therealfidus}
\nonumber &\lp \frac{1}{\sqrt{2 \pi}} \rp^{\nCr{N}{2}} \bigl( 1 + p \lp n, N \rp \bigr) \lp \int_0^\infty \cdots \int_0^\infty f_0 \lp s_1, \ldots, s_\lambda \rp \hsp d\textbf{s} \rp  \\ 
&\quad \times \int_{-\infty}^\infty \cdots \int_{-\infty}^\infty \lp 1 + h \lp \tilde{\alpha}_{12}, \ldots, \tilde{\alpha}_{N-1,N} \rp \rp f_1 \lp \left| \tilde{\alpha}_{12} \right|, \ldots, \left| \tilde{\alpha}_{N-1,N} \right| \rp \prod_{i,j = 1 \atop i < j}^N  \exp \lp -\tilde{\alpha}_{ij}^2 / 2 \rp  \hsp d \textbf{a}.
\end{align}
Now, define the number $D(n,N,K,f_1)$ by
\begin{align*}
D(n,N,K,f_1) &= \lp \int_{-\infty}^\infty \cdots \int_{-\infty}^\infty f_{1} \lp \left| \tilde{\alpha}_{12} \right|, \ldots, \left| \tilde{\alpha}_{N-1,N} \right| \rp \prod_{i,j = 1 \atop i < j}^N \exp \lp -\tilde{\alpha}_{ij}^2 / 2 \rp  \hsp d \textbf{a} \rp^{-1} 
\\ &\quad \times \int_{-\infty}^\infty \cdots \int_{-\infty}^\infty h \lp \tilde{\alpha}_{12}, \ldots, \tilde{\alpha}_{N-1,N} \rp f_{1} \lp \left| \tilde{\alpha}_{12} \right|, \ldots, \left| \tilde{\alpha}_{N-1,N} \right| \rp \prod_{i,j = 1 \atop i < j}^N \exp \lp -\tilde{\alpha}_{ij}^2 / 2 \rp  \hsp d \textbf{a},
\end{align*}
which is well-defined since $\int f_1 > 0$. It follows that we now have
\begin{align}\label{veganfuckheads}
\nonumber &\int_{-\infty}^\infty \cdots \int_{-\infty}^\infty \lp 1 + h \lp \tilde{\alpha}_{12}, \ldots, \tilde{\alpha}_{N-1,N} \rp \rp f_{1} \lp \left| \tilde{\alpha}_{12} \right|, \ldots, \left| \tilde{\alpha}_{N-1,N} \right| \rp \prod_{i,j = 1 \atop i < j}^N \exp \lp -\tilde{\alpha}_{ij}^2 / 2 \rp  \hsp d \textbf{a} 
\\ &\quad = \lp 1 + D(n,N,K,f_1) \rp \int_{-\infty}^\infty \cdots \int_{-\infty}^\infty f_{1} \lp \left| \tilde{\alpha}_{12} \right|, \ldots, \left| \tilde{\alpha}_{N-1,N} \right| \rp \prod_{i,j = 1 \atop i < j}^N \exp \lp -\tilde{\alpha}_{ij}^2 / 2 \rp  \hsp d \textbf{a} .
\end{align}
Finally, we define
\begin{align*}
C(n, N, K, f_1) := D(n,N,K,f_1) + p(n,N) + D(n,N,K,f_1) p(n,N)
\end{align*}
and note that 
\begin{align*}
\bigl| C(n, N, K, f_1) \bigr| \ll \bigl| D(n, N, K, f_1) \bigr| \ll_K \frac{N^3}{\sqrt{n}}
\end{align*}
where the implied constant is independent of $f_1$. The last inequality follows from Lemma \ref{p1-lemma2.5} and the fact that in the integrals that appear in the definition of $D(n, N, K, f_1)$, the exponential functions and $f_1$ are non-negative. Then, (\ref{p1-thm4.1-4}), (\ref{p1-therealfidus}), and (\ref{veganfuckheads}) show that the expectation on the left-hand side in (\ref{p1-thm4.1-1}) equals
\begin{align*}
&\hphantom{=} \bigl( 1 + C(n, N, K, f_1) \bigr) \frac{1}{2^\lambda} \left( \frac{1}{\sqrt{2 \pi}} \right)^{\nCr{N}{2}} \int_0^\infty \cdots \int_0^\infty \int_{-\infty}^\infty \cdots \int_{-\infty}^\infty 
 f \lp s_1, \ldots, s_\lambda, |\tilde{\alpha}_{12}|, \ldots, |\tilde{\alpha}_{N-1,N}| \rp \\ &\quad \quad \times \prod_{i,j = 1 \atop i < j}^N \exp \lp - \tilde{\alpha}_{ij}^2 / 2 \rp \hsp  d \textbf{a} d \textbf{s} +O\left(2^{-\lambda} 5^{\lfloor \lambda^2/4 \rfloor} \lp \sqrt{3}/2 \rp^n M (K+1)^\lambda\right)
\\ &= \bigl( 1 + C(n,N,K,f_1) \bigr) \frac{1}{2^\lambda} \left( \frac{2}{\pi} \right)^{\frac{1}{2} \nCr{N}{2}} \int_0^\infty \cdots \int_0^\infty \int_{0}^\infty \cdots \int_{0}^\infty f (s_1, \ldots, s_\lambda, \eta_{12}, \ldots, \eta_{N-1,N}) \\ &\quad \quad \times \prod_{i,j = 1 \atop i < j}^N \exp \lp - \eta_{ij}^2 / 2 \rp \hsp  d \eta_{N-1, N} \cdots d \eta_{12} d \textbf{s} + O\left(2^{-\lambda} 5^{\lfloor \lambda^2/4 \rfloor} \lp \sqrt{3}/2 \rp^n M (K+1)^\lambda\right).\\[-1.3em]
\end{align*}
Finally, we apply Campbell's theorem in the form given in \cite[Lemma 4.2]{soda10} to finish the proof.
 \end{proof} 
\section{Proof of Theorem \ref{p1-maintheorem}}
The aim of this section is to use Theorem \ref{p1-maintechnical} and a sieving argument to prove Theorem \ref{p1-maintheorem}.
\par We note that for the purpose of computing expected values, we may assume, with no loss of generality, that the random lattices $L$ under consideration satisfy $\mathcal{V}_1 < \mathcal{V}_2 < \mathcal{V}_3 < \cdots$, since in any dimension, the set $Z_n$ of all such unimodular lattices has full measure, cf. \cite[Lemma 5.1]{soda10}. \par 
We recall that for any $\lambda \geqslant 2$, the set $M_\lambda$ is defined as
\begin{align*}
M_\lambda = \left\lbrace \textbf{n} = \lp n_1, n_2, \ldots, n_\lambda \rp \in \Z_+^\lambda : n_a = n_b \Leftrightarrow a = b\right\rbrace.
\end{align*}
For $L \in X_n$, we will also write
\begin{align*}
N_n (L, x) := \# \left\lbrace j : \mathcal{V}_j \leqslant x \right\rbrace, \quad \quad N_\infty (\lb a, b \rb) := \# \left\lbrace j : a \leqslant T_j \leqslant b \right\rbrace, \quad \quad N_\infty (x) := N_\infty (\lb 0, x \rb).
\end{align*}
Let $\ell \geqslant 0$ be an integer, and define the random variable $R_\ell^n$ on $X_n$ by
\begin{align}
\begin{aligned}\label{p1-Rln}
R_\ell^n (L) &= \sum_{\textbf{n} \in M_{N + \ell}} f \lp \mathcal{V}_{n_1}, \ldots, \mathcal{V}_{n_N}, \tilde{\varphi}_{n_1 n_2}, \ldots, \tilde{\varphi}_{n_{N-1} n_N} \rp I\lp \mathcal{V}_{n_1} < \cdots < \mathcal{V}_{n_N} \text{ and } \mathcal{V}_{n_{N+1}} < \cdots < \mathcal{V}_{n_{N+ \ell}} < \mathcal{V}_{n_N} \rp
\\ &\stackrel{\text{a.e.}}{=} \sum_{\textbf{n} \in M_N} \nCr{n_N - N}{\ell} I(\mathcal{V}_{n_1} < \cdots < \mathcal{V}_{n_N}) f \lp \mathcal{V}_{n_1}, \ldots, \mathcal{V}_{n_N}, \tilde{\varphi}_{n_1 n_2}, \ldots, \tilde{\varphi}_{n_{N-1}n_N} \rp,
\end{aligned} \hspace{0.5cm} 
\end{align}
where the last equality holds for all lattices $L \in Z_n$. Also let
\begin{align}
\begin{aligned}\label{p1-Sln}
S_\ell^n (L) &= \sum_{j = 0}^\ell (-1)^j R_j^n (L) 
\\ &\stackrel{\text{a.e.}}{=} \sum_{\textbf{n} \in M_N} f \lp \mathcal{V}_{n_1}, \ldots, \mathcal{V}_{n_N}, \tilde{\varphi}_{n_1 n_2}, \ldots, \tilde{\varphi}_{n_{N-1}n_N} \rp I(\mathcal{V}_{n_1} < \cdots < \mathcal{V}_{n_N}) \sum_{j= 0}^\ell (-1)^j \nCr{n_N-N}{j},
\end{aligned}
\end{align}
where, again, the last inequality holds for all lattices $L \in Z_n$. We define $R_\ell^\infty$ by replacing each $\mathcal{V}_j$ and $\tilde{\varphi}_{n_s n_t}$ in (\ref{p1-Rln}) with $T_j$ and $\Phi_{n_s n_t}$, respectively. We then define $S_\ell^\infty$ similarly as in (\ref{p1-Sln}), but with $R_j^\infty$ in place of $R_j^n (L)$.\par
By applying Theorem \ref{p1-maintechnical} and writing
\begin{align*}
e(\lambda, n , K, M) = 2^{-\lambda} 5^{\lfloor \lambda^2 / 4 \rfloor} (\sqrt{3}/2)^n M (K+1)^\lambda,
\end{align*}
we get that
\begin{align}\label{p1-part1}
\nonumber \expect \lb S_\ell^n (L) \rb &= \sum_{j = 0}^\ell (-1)^j \expect \lb R_j^n (L) \rb 
\\ \nonumber &= \sum_{j = 0}^\ell (-1)^j  \ \bigl( 1 + C(n, N, K, f_1) \bigr) \expect \lb R_j^\infty \rb + O \left( \sum_{j = 0}^\ell e(N+j, n, K, M) \right)
\\ &=  \bigl( 1 + C(n, N, K, f_1) \bigr) \expect \lb S_\ell^\infty \rb + O \left( \sum_{j = 0}^\ell e(N+j, n, K, M) \right).
\end{align}
To bound the last term in (\ref{p1-part1}), we observe that
\begin{align*}
\sum_{j = 0}^\ell e(N+j, n, K,M) &=  \left( \frac{\sqrt{3}}{2} \right)^n M \left( \frac{K+1}{2} \right)^N   \sum_{j = 0}^\ell \left( \frac{K+1}{2} \right)^j 5^{\lfloor (N+j)^2 /4 \rfloor}\\
&\ll  5^{(N+\ell)^2 / 4}  \left( \frac{\sqrt{3}}{2} \right)^n M \left( \frac{K+1}{2} \right)^{N+\ell} \\
&\ll \left( \frac{3}{2} \right)^{(N+\ell)^2} \left( \frac{\sqrt{3}}{2} \right)^n M \left( \frac{K+1}{2} \right)^{N+\ell} \\
&\ll \left( \frac{\sqrt{3}}{2} \right)^n \hspace{-0.1cm} M  K^{(N+ \ell)^2},
\end{align*}
since $K \geqslant 3$. It follows that 
\begin{align}\label{p1-partypolser}
\expect \lb S_\ell^n (L) \rb &= \bigl( 1 + C(n, N, K, f_1) \bigr) \expect \lb S_\ell^\infty \rb + O_M \lp \left( \frac{\sqrt{3}}{2} \right)^n \hspace{-0.1cm} K^{(N+ \ell)^2} \rp.
\end{align}
\par If we let
\begin{align*}
X(L, \ell, N) &= \sum_{\textbf{n} \in M_N} f \lp \mathcal{V}_{n_1}, \ldots, \mathcal{V}_{n_N}, \tilde{\varphi}_{n_1 n_2}, \ldots, \tilde{\varphi}_{n_{N-1} n_N} \rp I(\mathcal{V}_{n_1} < \cdots < \mathcal{V}_{n_N} ){{n_N-N-1}\choose{\ell}},
\end{align*}
then provided that $\ell$ is even, by using the explicit expression for an alternating sum of binomial coefficients given in \cite[Sect. 5]{soda10}, we obtain that
\begin{align}\label{p1-nice2}
\nonumber S_\ell^n (L) &= \sum_{\textbf{n} \in M_N} f \lp \mathcal{V}_{n_1}, \ldots, \mathcal{V}_{n_N}, \tilde{\varphi}_{n_1 n_2}, \ldots, \tilde{\varphi}_{n_{N-1} n_N} \rp I(\mathcal{V}_{n_1} < \cdots < \mathcal{V}_{n_N}) \sum_{j = 0}^\ell (-1)^j {{n_N - N}\choose{j}} \\
&= f\left(\mathcal{V}_1, \ldots, \mathcal{V}_N, \tilde{\varphi}_{12},\ldots, \tilde{\varphi}_{N-1,N} \right) + X(L, \ell, N)
\end{align}
for all $L \in Z_n$. Similarly, if $\ell$ is odd, then for all $L \in Z_n$,
\begin{align}\label{p1-nice1}
S_\ell^n (L) = f\left(\mathcal{V}_{1}, \ldots, \mathcal{V}_{N}, \tilde{\varphi}_{12}, \ldots, \tilde{\varphi}_{N-1,N} \right) - X(L, \ell, N).
\end{align}
Consequently, for all $\ell$ and all $L \in Z_n$,
\begin{align*}
S_\ell^n (L ) = f\left(\mathcal{V}_{1}, \ldots, \mathcal{V}_{N}, \tilde{\varphi}_{12}, \ldots, \tilde{\varphi}_{N-1,N} \right) \pm X(L, \ell, N).
\end{align*}
Analogously, we also have (with probability $1$)
\begin{align}\label{p1-niceinfinity}
S_\ell^\infty = f\left(T_1, \ldots, T_N, \Phi_{12}, \ldots, \Phi_{N-1,N} \right) \pm X^\infty(\ell, N)
\end{align}
with
\begin{align*}
0 &\leqslant X^\infty(\ell, N) 
\\ &= \sum_{\textbf{n} \in M_N} f\left(T_{n_1}, \ldots, T_{n_N}, \Phi_{n_1 n_2}, \ldots, \Phi_{n_{N-1} n_N}\right) I(T_{n_1} < \cdots < T_{n_N}){{n_N-N-1}\choose{\ell}}\\
&\leqslant \sum_{\textbf{n} \in M_N} f \lp T_{n_1}, \ldots, T_{n_N}, \Phi_{n_1 n_2}, \ldots, \Phi_{n_{N-1} n_N} \rp I(T_{n_1} < \cdots < T_{n_N} \leqslant K){{N_\infty (K)-N-1}\choose{\ell}}
\end{align*}
where we used that $f$ is non-negative and vanishes outside $\lb 0,K \rb^{N+\nCr{N}{2}}$, and that on the support of the indicator function, $n_N = N_\infty \lp T_{n_N} \rp \leqslant N_\infty (K)$.
\par We want to further estimate $X^\infty$. To this end, we will need the following estimate. Suppose that $\mathsf{X}$ is a Poisson distributed random variable with mean $\lambda$ and that $x > \lambda$ is any number. Then 
\begin{align}\label{p1-poissonestimate}
\mathbb{P}\bigl( \mathsf{X} \geqslant x \bigr) \leqslant e^{-\lambda} \left( \frac{e \lambda}{x} \right)^x.
\end{align}
To see this, note that for any $t > 0$, Markov's inequality gives
\begin{align*}
\mathbb{P}\bigl( \mathsf{X} \geqslant x \bigr) = \mathbb{P}\lp e^{t\mathsf{X}} \geqslant e^{tx} \rp \leqslant e^{-tx} \expect \lb e^{t\mathsf{X}} \rb = e^{\lambda(e^t - 1)-tx}.
\end{align*}
Taking $t = \log \lp x/ \lambda \rp > 0$ now proves the claim. Towards estimating $X^\infty$, we note that 
\begin{align*}
\nCr{N_\infty(K) - N -1}{\ell} \leqslant 2^{N_\infty (K) - N - 1} < 2^{N_\infty(K) - N},
\end{align*}
and that $X^\infty$ vanishes if $\ell > N_\infty (K) - N - 1$. Recalling that
\begin{align*}
G_f (N) = \expect \lb f_1 \lp \Phi_{12}, \ldots, \Phi_{N-1,N} \rp \rb,
\end{align*}
we consequently have the estimate
\begin{align}\label{p1-estimatethis}
\nonumber \expect \lb X^\infty (\ell, N) \rb &\leqslant \sqrt{M} \sum_{\textbf{n} \in M_N} \expect \lb I \lp T_{n_1} < \cdots < T_{n_N} \leqslant K \rp 2^{N_\infty (K) - N} \right. 
\\ \nonumber &\quad \quad \times f_1 \lp \Phi_{n_1 n_2}, \ldots, \Phi_{n_{N-1}n_N} \rp I \bigl( N_\infty (K) \geqslant N + \ell + 1 \bigr) \Big] 
\\ \nonumber &= \sqrt{M}  G_f(N) \sum_{\textbf{n} \in M_N} \expect \lb I \lp T_{n_1} < \cdots < T_{n_N} \leqslant K \rp 2^{N_\infty (K) - N}  I \bigl( N_\infty (K) \geqslant N + \ell + 1 \bigr) \rb
\\ \nonumber &= \sqrt{M}  G_f(N)  \expect \lb \nCr{N_\infty (K)}{N} 2^{N_\infty (K) - N}  I\left(N_\infty(K) \geqslant N + \ell + 1 \right) \rb
\\ \nonumber &\ll_{K,M} \, G_f(N) \hspm \hspm \hspm  \sum_{s = N + \ell + 1}^\infty \nCr{s}{N} 2^{s - N} \frac{(K/2)^s}{s!} \leqslant 2^{-N} G_f(N)  \hspm \hspm \hspm \sum_{s = N + \ell + 1}^\infty \frac{(2K)^s}{s!} 
\\ &\leqslant 2^{-N} G_f(N) \left( \frac{2 e K}{N + \ell + 1} \right)^{N + \ell + 1},
\end{align}
where we used the fact that $N_\infty (K)$ is Poisson distributed with mean $K/2$, the estimate (\ref{p1-poissonestimate}) with $\lambda = 2K$, and the fact that the Gaussian variables are independent of the Poisson process and of each other.
\par Now (\ref{p1-niceinfinity}) and (\ref{p1-estimatethis}) imply that
\begin{align}\label{p1-hellsyeahbrotherman}
\expect \lb S_\ell^\infty \rb &= \expect \lb f \left(T_1, \ldots, T_N, \Phi_{12}, \ldots, \Phi_{N-1,N} \right) \rb + O_{K,M} \left( 2^{-N} G_f(N) \left( \frac{2 e K}{N + \ell + 1} \right)^{N + \ell + 1}  \right),
\end{align}
and (\ref{p1-partypolser}) and (\ref{p1-hellsyeahbrotherman}) imply that
\begin{align}\label{p1-wutang}
\nonumber \expect \lb S_\ell^n (L) \rb &= \bigl( 1 + C(n, N, K, f_1 ) \bigr) \expect \lb f \left(T_1, \ldots, T_N, \Phi_{12}, \ldots, \Phi_{N-1,N} \right) \rb  
\\ &\quad \quad +  O_{K,M} \left( 2^{-N} G_f(N) \left( \frac{2 e K}{N + \ell + 1} \right)^{N + \ell + 1}  \right) + O_M \left( \left( \frac{\sqrt{3}}{2} \right)^n \hspace{-0.1cm} K^{(N+ \ell)^2} \right).
\end{align}
Now let 
\begin{align*}
\ell_0 \lp N^2 \rp &= \max \left\lbrace m \in \mathbb{Z} : m \text{ is even and }m \leqslant N^2 \right\rbrace, \\
\ell_1 \lp N^2 \rp &= \max \left\lbrace m \in \mathbb{Z} : m \text{ is odd and } m \leqslant N^2 \right\rbrace.
\end{align*}
If $\ell = N^2$ or $\ell = N^2 - 1$, (\ref{p1-wutang}) implies that
\begin{align}\label{p1-dutang}
\nonumber \expect \lb S_\ell^n (L) \rb &= \bigl( 1 + C(n, N, K, f_1 ) \bigr) \expect \lb f \left(T_1, \ldots, T_N, \Phi_{12}, \ldots, \Phi_{N-1,N} \right) \rb  
\\ &\quad \quad +  O_{K,M} \left( 2^{-N} G_f(N) \left( \frac{2 e K}{N^2} \right)^{N^2}  \right) + O_M \left( \left( \frac{\sqrt{3}}{2} \right)^n \hspace{-0.1cm} K^{4N^4} \right).
\end{align}
Finally, (\ref{p1-nice2}) and (\ref{p1-nice1}) imply that
\begin{align*}
\expect \lb S_{\ell_1 (N^2)}^n (L) \rb \leqslant \expect \lb f\bigl( \mathcal{V}_1, \ldots, \mathcal{V}_N, \tilde{\varphi}_{12}, \ldots, \tilde{\varphi}_{N-1,N} \bigr) \rb \leqslant \expect \lb S_{\ell_0 (N^2)}^n (L) \rb, 
\end{align*}
and we conclude that
\begin{align}\label{p1-fuckyeahhhh}
\nonumber \expect \lb f\left(\mathcal{V}_1, \ldots, \mathcal{V}_n, \tilde{\varphi}_{12}, \ldots, \tilde{\varphi}_{N-1,N}\right) \rb &= \bigl( 1 + C(n, N, K, f_1 ) \bigr) \expect \lb f \left(T_1, \ldots, T_N, \Phi_{12}, \ldots, \Phi_{N-1,N} \right) \rb   \\ &\hspace{1cm} + O_{K,M} \left( 2^{-N} G_f(N) \left( \frac{2 e K}{N^2} \right)^{N^2} + \lp \frac{\sqrt{3}}{2} \right)^n \hspace{-0.1cm} K^{4N^4} \right).
\end{align}
Finally, since 
\begin{align*}
C(n, N, K, f_1) = O_{K} \lp N^3 / \sqrt{n} \rp,
\end{align*}
(\ref{p1-fuckyeahhhh}) completes the proof of Theorem \ref{p1-maintheorem}. \\\\
\textsc{Remark.} We stress that, although we certainly have 
\begin{align*}
\lp \frac{\sqrt{3}}{2} \right)^n \hspace{-0.1cm} K^{4N^4} \ll 2^{-N} \left( \frac{2 e K}{N^2} \right)^{N^2},
\end{align*}
the factor $G_f(N)$ could potentially be very small, thus making $\lp \sqrt{3} / 2 \right)^n K^{4N^4}$ the larger of the two error terms.\\ \par 
We now demonstrate that Theorem \ref{p1-maintheorem} gives a meaningful statement in (at least) the case mentioned in the introduction; namely when $f_0$ is the indicator function of the set $\lb 0, K \rb^{N}$, and $f_1$ is the indicator function of the Cartesian product $D_{12} \times \cdots \times D_{N-1,N}$ where the sets $D_{ij} \subset \lb 0, K \rb$ are arbitrary measurable sets of Lebesgue measure at least $\xi > 0$.\par 
We first note that the condition on the sets $D_{ij}$ guarantees that
\begin{align*}
G_f (N) = \left( \frac{2}{\pi} \right)^{\frac{1}{2}\nCr{N}{2}}\prod_{i, j = 1 \atop i < j}^N \int_{D_{ij}} e^{-x^2/2} \hspace{0.1cm} dx  \geqslant \lp \sqrt{ \frac{2}{\pi} } \int_{K- \xi}^K e^{-x^2 / 2} \hsp dx \rp^{b(N)} =: \, \delta ( \xi, K)^{b(N)}.
\end{align*}
We now see that 
\begin{align*}
\expect \lb f\left(T_1, \ldots, T_N, \Phi_{12}, \ldots , \Phi_{N-1,N} \right) \rb &=  G_f (N)\mathbb{P} \bigl( 0 \leqslant T_1 < \cdots < T_N \leqslant K \bigr)\\[+1em]
&= e^{-K/2}  G_f (N)  \sum_{s \geqslant N} \frac{(K/2)^s}{s!}   \\
&\gg_K   N^{-\frac{1}{2}} G_f (N) \left( \frac{eK}{2N} \right)^N 
\end{align*}
by Stirling's formula. It is clear that we have
\begin{align*}
N^{-1/2} \left( \frac{eK}{2N} \right)^N G_f (N) \gg 2^{-N} G_f (N) \lp \frac{2eK}{N^2} \rp^{N^2}.
\end{align*}
Moreover, the bound $G_f (N) \geqslant \delta ( \xi, K)^{b(N)}$ shows that
\begin{align*}
N^{-1/2} \left( \frac{eK}{2N} \right)^N G_f (N) \gg \lp \frac{\sqrt{3}}{2} \rp^n K^{4N^4},
\end{align*}
and the claim follows.
\section{The Normalization of the Angles}
In this final section, we prove the following result, which states that the normalization (\ref{p1-disdanormalization}) is still natural in the case of an increasing positive number $N = o \lp n^{1/6} \rp$ of short lattice vectors.  We recall that for $L \in X_n$, $\textbf{v}_i$ denotes (a representative of the class of) the $i$'th shortest non-zero vector in $L$, and that $\varphi \lp \textbf{v}_i , \textbf{v}_j \rp$ denotes the angle between $\pm \textbf{v}_i$ and $\pm \textbf{v}_j$ taken in $\lb 0, \tfrac{\pi}{2} \rb$. 
\begin{prop}\label{p1-anglenormalization-prop}
\textit{Let $N = N(n)$ be an integer with $N \longrightarrow \infty$ as $n \longrightarrow \infty$  and $N = o \lp n^{1/6} \rp$. Let $\psi (N) \ll N^2 $ be any function such that
\begin{align*}
\lim_{N \rightarrow \infty} \frac{\psi \lp N \rp}{\log N} = \infty.
\end{align*}
Let $\varepsilon > 0$ and let $C = C(N)$ be any function satisfying}
\begin{align*}
\lp \log N  + \log \psi(N) \rp^{1/2 + \varepsilon} \ll C(N).
\end{align*}
\textit{Then one has the estimate}
\begin{align}\label{p1-prop6.1statement}
\mu \lp \lbr L \in X_n : \frac{\pi}{2} - \varphi \lp \textbf{v}_i, \textbf{v}_j \rp \leqslant \frac{C}{\sqrt{n}} \text{ for all } 1 \leqslant i < j \leqslant N \rbr \rp = 1 + o_n(1)
\end{align}
\textit{as} $n \longrightarrow \infty$.
\end{prop}
\noindent \textsc{Remark.} We stress that in relation to the normalization (\ref{p1-disdanormalization}), we are primarily interested in small choices of $C$.
\begin{proof}
For technical reasons, we first prove the proposition with the additional assumption that 
\begin{align}\label{p1-upperboundonC}
C \ll \frac{N}{\sqrt{\psi(N)}}.
\end{align}
Given $V > 0$, let $B_V$ denote the closed ball in $\R^n$ with volume $V$, which is centered at the origin. Also, for $0 \leqslant \varphi_1 < \varphi_2 \leqslant \tfrac{\pi}{2}$, we define
\begin{align*}
I \lp \varphi_1, \varphi_2, B_V, \textbf{x}, \textbf{y} \rp = 
\begin{cases}
1, &\text{if } \textbf{x} \neq \pm \textbf{y}, \hsp \hsp \textbf{x}, \textbf{y} \in B_V, \, \text{ and } \varphi \lp \textbf{x}, \textbf{y} \rp \in \lb \varphi_1, \varphi_2 \rb; \\
0, &\text{otherwise.}
\end{cases}
\end{align*}
We additionally let
\begin{align*}
M_{V, \varphi_1, \varphi_2} (L) = \frac{1}{8} \sum_{\textbf{m}_1 \in L' \atop \textbf{m}_2 \in L'} I \lp \varphi_1, \varphi_2, B_V, \textbf{m}_1, \textbf{m}_2 \rp,
\end{align*}
so that $M_{V, \varphi_1, \varphi_2} (L)$ counts the number of (representatives of) pairs of non-zero lattice vectors of $L \cap B_V$ whose angle lies between $\varphi_1$ and $\varphi_2$.\par 
We proceed as in \cite[Sect. 2]{soda10} and apply Rogers' formula \cite[Thm. 4]{rogers55} in combination with the first estimate on p. 247 in \cite{rogers55v2} to obtain the estimate
\begin{align*}
\expect \lb M_{V, \varphi_1, \varphi_2} \rb &= \frac{1}{8} \int_{\R^n} \int_{\R^n} I \lp \varphi_1, \varphi_2, B_V, \textbf{x}, \textbf{y} \rp \hsp d \textbf{x} \, d \textbf{y} + O \lp 2^{-n} V \rp.
\end{align*}
Furthermore, by changing to spherical coordinates in the integral, we see that 
\begin{align*}
\expect \lb M_{V, \varphi_1, \varphi_2} \rb &= \frac{V^2}{4} \frac{\omega_{n-1}}{\omega_n} \int_{\varphi_1}^{\varphi_2} \sin \lp \phi \rp^{n-2} \hsp d\phi + O \lp 2^{-n} V \rp
\end{align*}
cf. \cite[Eq. (2.3)]{soda10}. By this, and by linearity of expectation, we see that
\begin{align}\label{p1-linofexp}
\nonumber \expect \lb M_{V, 0, \pi/2 - C/\sqrt{n}} \rb 
&= \expect \lb M_{V, 0, \pi/2} \rb - \expect \lb M_{V, \pi/2 - C/\sqrt{n}, \pi / 2} \rb \\
&= \frac{V^2}{4} \frac{\omega_{n-1}}{\omega_n} \lp \int_{0}^{\pi/2} \sin \lp \phi \rp^{n-2} \hsp d\phi - \int_{\pi/2-C/\sqrt{n}}^{\pi/2} \sin \lp \phi \rp^{n-2} \hsp d\phi \rp + O \lp 2^{-n} V \rp.
\end{align}
We now inspect the two integrals individually. To deal with the first one, we use the fact that if $\Gamma$ denotes the Gamma function, then
\begin{align*}
2 \int_0^{\pi /2} \sin(u)^{2x-1} \cos(u)^{2y-1} \hsp du = \frac{\Gamma \lp x \rp \Gamma \lp y \rp}{\Gamma \lp x + y \rp}
\end{align*}
(cf. \cite[Thm. 2.5, Thm. 2.7]{bell}), from which it follows that
\begin{align}\label{p1-wallisisgamma}
\int_{0}^{\pi/2} \sin \lp \phi \rp^{n-2} \hsp d\phi = \frac{\sqrt{\pi}}{2} \frac{\Gamma \lp \tfrac{n-1}{2} \rp}{\Gamma \lp \tfrac{n}{2} \rp}.
\end{align}
Applying Stirling's formula, we obtain
\begin{align*}
\frac{\Gamma \lp \tfrac{n-1}{2} \rp}{\Gamma \lp \tfrac{n}{2} \rp} 
&=  \lp 1 + O \lp \frac{1}{n} \rp \rp \frac{\lp \tfrac{n-1}{2} \rp^{n/2 - 1} \exp \lp \tfrac{n}{2} \rp }{\lp \tfrac{n}{2} \rp^{(n-1)/2} \exp \lp \tfrac{n-1}{2} \rp} \\
&= \lp 1 + O \lp \frac{1}{n} \rp \rp \frac{\sqrt{e}}{\sqrt{2}} \frac{2 \sqrt{n}}{n-1}  \lp \frac{n-1}{n} \rp^{n/2} \\
&= \lp 1 + O \lp \frac{1}{n} \rp \rp\frac{\sqrt{2n}}{n-1}.
\end{align*}
Together with (\ref{p1-quotient}), (\ref{p1-wallisisgamma}) now proves that
\begin{align}\label{p1-integral1}
\frac{\omega_{n-1}}{\omega_n} \int_{0}^{\pi/2} \sin \lp \phi \rp^{n-2} \hsp d\phi &= \lp 1 + O \lp \frac{1}{n} \rp \rp \frac{1}{2} \sqrt{\frac{n}{n-1}} = \frac{1}{2} + O \lp \frac{1}{n} \rp.
\end{align}
As for the second integral, we see that 
\begin{align*}
\frac{\omega_{n-1}}{\omega_n} \int_{\pi/2-C/\sqrt{n}}^{\pi/2} \sin \lp \phi \rp^{n-2} \hsp d\phi 
= \frac{\omega_{n-1}}{\omega_n \sqrt{n}} \int_0^C \cos \lp \frac{t}{\sqrt{n}} \rp^{n-2} \hsp dt.
\end{align*}
By Taylor estimates, we have 
\begin{align*}
\cos \lp \frac{t}{\sqrt{n}} \rp^{n-2} = \exp \lp -\frac{t^2}{2} \rp \lp 1 + O \lp \frac{C^4}{n} \rp \rp
\end{align*}
on the interval $\lb 0, C \rb$. Hence, using again (\ref{p1-quotient}), we find that
\begin{align}\label{p1-integral2}
\frac{\omega_{n-1}}{\omega_n} \int_{\pi/2-C/\sqrt{n}}^{\pi/2} \sin \lp \phi \rp^{n-2} \hsp d\phi = \frac{1}{\sqrt{2 \pi}} \int_0^C \exp \lp -\frac{t^2}{2} \rp \hsp dt + O \lp \frac{C^4}{n} \rp = \frac{1}{2} \mathrm{erf} \lp \frac{C}{\sqrt{2}} \rp +  O \lp \frac{C^4}{n} \rp,
\end{align}
where, by the substitution $s = t/\sqrt{2}$, we expressed the integral in terms of the \textit{error function},
\begin{align*}
\mathrm{erf} (x) = \frac{2}{\sqrt{\pi}} \int_0^x \exp \lp -s^2 \rp \hsp ds.
\end{align*}
It follows from (\ref{p1-linofexp}), (\ref{p1-integral1}), and (\ref{p1-integral2}) that
\begin{align*}
\expect \lb M_{V, 0, \pi/2 - C/\sqrt{n}} \rb  &= \frac{V^2}{8}\lp 1 - \mathrm{erf} \lp \frac{C}{\sqrt{2}} \rp \rp + O \lp \frac{V^2 C^4}{n} \rp   + O \lp 2^{-n} V \rp.
\end{align*}
Furthermore, since
\begin{align*}
1 - \mathrm{erf} \lp \frac{C}{\sqrt{2}} \rp < \frac{2}{\sqrt{\pi}} \int_{C/\sqrt{2}}^\infty \frac{\sqrt{2} s}{C} \exp \lp -s^2 \rp \hsp ds  =  \frac{\sqrt{2}}{\sqrt{\pi}C} e^{-C^2 / 2},
\end{align*}
we have proved that with 
\begin{align*}
B(V) &= \lbr L \in X_n : \frac{\pi}{2} - \varphi ( \textbf{v} ,\textbf{w} ) \leqslant \frac{C}{\sqrt{n}} \text{ holds for all } \textbf{v}, \textbf{w} \in \lp B_V \cap L \rp \setminus \lbrace 0 \rbrace \rbr, \\ q \lp n, N, V, C \rp &= \frac{V^2}{C} e^{-C^2 / 2}  + \frac{V^2 C^4}{n} + 2^{-n} V,
\end{align*}
one has
\begin{align}\label{p1-measureofb}
\mu \bigl( B(V) \bigr) = 1 + O \bigl( q \lp n, N, V, C \rp \bigr).
\end{align}
\par Let us now write
\begin{align*}
A(V) = \Big\{ L \in X_n : \textbf{v}_1 \lp L \rp, \ldots, \textbf{v}_N \lp L \rp \in B_V \Big\}.
\end{align*}
If we assume that $V$ satisfies
\begin{align}\label{p1-V-conditionforKim}
V \leqslant \frac{1}{8} \lp \sqrt{\frac{n}{2}} - N \rp,
\end{align}
it follows from \cite[Thm. 4]{kim16} that if $V \geqslant 2$,
\begin{align}\label{p1-measureofa}
\nonumber \mu \lp A(V) \rp &= 1 - e^{-V/2} \sum_{i = 0}^{N-1} \frac{\lp V/2 \rp^i}{i!} + O\lp \lp \sqrt{\frac{n}{2}} - N \rp^{-1/2} \rp + O \lp \frac{0.999^n}{N!} \sqrt{\frac{n}{2}} \lp \frac{V}{2} + 1 \rp^{\sqrt{n/2}} \rp \\
&\gg 1 - e^{-V/2} N \lp \frac{V}{2} \rp^N - p \lp n, N, V \rp 
\end{align}
with 
\begin{align*}
p(n, N, V ) = \lp \sqrt{\frac{n}{2}} - N \rp^{-1/2} + 
\frac{0.999^n}{N!} \sqrt{\frac{n}{2}} \lp \frac{V}{2} + 1 \rp^{\sqrt{n/2}}.
\end{align*}
Finally, we define the set 
\begin{align*}
D = \lbr L \in X_n : \frac{\pi}{2} - \varphi \lp \textbf{v}_i, \textbf{v}_j \rp \leqslant \frac{C}{\sqrt{n}} \text{ holds for all } 1 \leqslant i < j \leqslant N \rbr.
\end{align*}
Then, by (\ref{p1-measureofb}), 
\begin{align*}
\mu \bigl( A(V) \bigr) &= \mu \bigl( A(V) \cap B(V) \bigr) + \mu \bigl( A(V) \cap B(V)^\complement \bigr) = \mu \bigl( A(V) \cap B(V) \bigr) + O \bigl( q \lp n,N, V,C \rp \bigr),
\end{align*}
and since $A(V) \cap B(V) \subset D$ for any $V$, it follows from (\ref{p1-measureofa}) that
\begin{align}\label{p1-probisalmost1}
\mu \lp D \rp \gg \mu \bigl( A(V) \bigr) - q \lp n, N, V, C \rp \gg 1 - e^{-V/2} N \lp \frac{V}{2} \rp^N - p(n, N, V )  - q(n, N, V, C) .
\end{align}
\par We now let $V = N \psi (N)$. It remains to prove that with $C$ in the given range, the right-hand side of (\ref{p1-probisalmost1}) is $1 + o_n(1)$. To this end, we note that (\ref{p1-V-conditionforKim}) is satisfied for $n$ large enough since $V \ll N^3$. It then follows from \cite[Thm. 4]{kim16} and the fact $N = o \lp n^{1/6} \rp$ that $p (n, N, V ) \longrightarrow 0$ as $n \longrightarrow \infty$. Thus, we only have to check that the following limits hold as $n \longrightarrow \infty$:
\begin{align}
2^{-n} V &\longrightarrow 0, \label{p1-check00}\\
e^{-V/2} N \lp \frac{V}{2} \rp^N &\longrightarrow 0, \label{p1-check0}\\
\frac{V^2}{C} e^{-C^2 / 2} &\longrightarrow 0, \label{p1-check1}\\
\frac{V^2 C^4}{n} &\longrightarrow 0. \label{p1-check2}
\end{align}
We see that (\ref{p1-check00}) holds since $V \ll N^3$ and $N = o \lp n^{1/6} \rp$. Also, (\ref{p1-check0}) is satisfied since 
\begin{align*}
e^{-V/2} N \lp \frac{V}{2} \rp^N = \exp \lp - \frac{V}{2} + \log N + N \log V - N \log 2 \rp
\end{align*}
and
\begin{align*}
\frac{V}{N \log V} = \frac{\psi (N)}{ \log N + \log \psi (N)} \gg \frac{\psi (N)}{\log N} \longrightarrow \infty.
\end{align*}
Next, (\ref{p1-check1}) is satisfied if $C^2$ dominates $\log V = \log N + \log \psi (N)$. This is true by assumption. Finally, (\ref{p1-check2}) is certainly satisfied if $V^2 C^4 \ll N^6$; that is, if $\psi (N)^2 C^4 \ll N^4$. This is also true by virtue of the additional assumption (\ref{p1-upperboundonC}). This proves the proposition under the assumption (\ref{p1-upperboundonC}). \par  
We may now extend the statement in the proposition to the full range of $C$ by noting that 
\begin{align*}
\mu \lp \lbr L \in X_n : \frac{\pi}{2} - \varphi \lp \textbf{v}_i, \textbf{v}_j \rp \leqslant \frac{C}{\sqrt{n}} \text{ for all } 1 \leqslant i < j \leqslant N \rbr \rp
\end{align*}
is visibly increasing in $C$. Hence, as the statement of the proposition holds for any choice of $C$ satisfying (\ref{p1-upperboundonC}) and the required lower bound, it also holds for any larger $C$. This completes the proof of Proposition \ref{p1-anglenormalization-prop}.
\end{proof}
\begin{appendices}
\section{Appendix: Proof of the Inclusion (\ref{p1-theultimatecontainment})}
It remains to prove the inclusion (\ref{p1-theultimatecontainment}) for $n$ large enough. We state the result as the following proposition, which therefore finishes the proof of Theorem \ref{p1-maintechnical}.
\begin{prop}\label{p1-Prop3.3}
\textit{Let $K > 0$ and let $N = o \left( n^{1/6} \right)$ be an integer. For $n$ sufficiently large, we have the inclusion} 
\begin{align*}
{\bf p} + n^{-\frac{1}{2}} \bigl[ -K, K \bigr]^{\nCr{N}{2}} \subset \Omega_N,
\end{align*}
where ${\bf p} = \lp \tfrac{\pi}{2}, \ldots, \tfrac{\pi}{2} \rp \in \R^{b(N)}$, $\Omega_N = J_N \lp (0, \pi )^{b(N)} \rp$, and 
\begin{align*}
J_N : \lp \phi_{12}, \ldots, \phi_{N-1,N} \rp \longmapsto \lp \alpha_{12}, \ldots, \alpha_{N-1,N} \rp
\end{align*}
is given by (\ref{p1-fuckedupfunction}).
\end{prop}
We have $N = o\left( n^{1/6} \right)$, meaning that eventually $n^{-1/2} \leqslant N^{-3}$. We can assume, with no loss of generality, that $N \longrightarrow \infty$ when $n \longrightarrow \infty$, as the case where $N$ is bounded is handled in \cite{soda10}. Then it will suffice to show that for $n \gg 1$,
\begin{align}\label{p1-translationtoN}
\textbf{p} + N^{-3} \bigl[ -K, K \bigr]^{\nCr{N}{2}} \subset \Omega_N.
\end{align}
We initially make the following observations: First of all, for $ i < j, j'$ we have an equality of functions $\alpha_{ij} = \alpha_{ij'}$ (as defined in (\ref{p1-fuckedupfunction})) in the sense that, for any values of $\phi_{1i}, \phi_{1j}, \ldots, \phi_{ij}$,
\begin{align*}
\alpha_{ij}\lp \phi_{1i}, \phi_{1j}, \ldots, \phi_{ij} \rp = \alpha_{ij'}\lp \phi_{1i}, \phi_{1j}, \ldots, \phi_{ij} \rp.
\end{align*}
In particular, to understand the image of the function $\alpha_{ij}$, we only need to understand the image of $\alpha_{i,i+1}$. Second, if $i > 1$, then for fixed $\phi_{1i}, \ldots, \phi_{i-1,j}$, the function $\alpha_{ij}$ is increasing in $\phi_{ij}$, being a composition of two decreasing functions. \par 
For the remainder of the proof, we fix some notation. Let $I = (0, \pi)$, and for $s > 0$, let $I_s = \lb \tfrac{\pi-s}{2},  \tfrac{\pi+s}{2} \rb$. For $N \geqslant 3$, we say that $s \in \smash{\lp 0, \tfrac{\pi}{2} \rp}$ is \textit{\textbf{N-good}} if, for any $x_1, \ldots, x_{2(N-2)} \in I_s$ we have the inclusion
\begin{align*}
I_s \subset \alpha_{N-1,N} \left( \left( \prod_{m = 1}^{2(N-2)} \lbrace x_m \rbrace \right) \times I \right).
\end{align*}
It is worthwhile explaining the relevance of this notion, and why it is natural to consider in the situation at hand. Our current objective is to prove the existence of a hypercube of a certain magnitude and position inside the high-dimensional set $\Omega_N$. As it is impossible to visually inspect $\Omega_N$ for dimensional reasons, a natural idea is to use the fact that the diameter of a high-dimensional object can be estimated from below if one has sufficiently accurate information about the diameters of all of its possible lower-dimensional cross-sections, cf. Lemma \ref{p1-lemma3.7} below. The relevance of being $N$-good, then, is that an $N$-good number $s$ is precisely so small that a hypercube centered at $\lp \tfrac{\pi}{2}, \ldots, \tfrac{\pi}{2} \rp$ with sidelength $s$ is contained in the graph
\begin{align*}
I_s \times \cdots \times I_s \times I \times \alpha_{N-1, N} \lp I_s \times \cdots \times I_s \times I  \rp \subset \R^{2N-2}.
\end{align*}
We now give a criterion for being $N$-good. 
\begin{lemma}\label{p1-lemma3.4}
\textit{In order for $s$ to be $N$-good, it is sufficient that }
\begin{align*}
E(N-2,s) := (N-2) \sin \left( \frac{s}{2} \right)^2 - \cos \left( \frac{s}{2} \right)^{2(N-2)} + \sin \left( \frac{s}{2} \right) \leqslant 0.
\end{align*}
\end{lemma}
\begin{proof}
Let $x_1, \ldots, x_{2(N-2)} \in I_s$ be arbitrary. The continuity of $\alpha_{N-1,N}$ means that it will suffice to show that the hypothesis $E(N-2, s) \leqslant 0$ implies the inequalities
\begin{align*}
\sup_{x \in I} \alpha_{N-1,N} \lp x_1, \ldots, x_{2(N-2)}, x \rp &\geqslant \max I_s = \frac{\pi + s}{2}, 
\\
 \inf_{x \in I} \alpha_{N-1,N} \lp x_1, \ldots, x_{2(N-2)}, x \rp &\leqslant \min I_s = \frac{\pi - s}{2}.
\end{align*}
Since $\alpha_{N-1,N}$ is increasing in $\phi_{N-1,N}$, it follows that
\begin{align*}
\cos \left( \sup_{x \in I} \alpha_{N-1,N} \lp x_1, \ldots, x_{2(N-2)}, x \rp \right) &= F \lp  x_1, \ldots, x_{2(N-2)} \rp - X \lp x_1, \ldots, x_{2(N-2)} \rp \\
&\leqslant \sup \lbr F(\textbf{x}) : \textbf{x} \in I_s^{2(N-2)} \rbr - \inf \lbr X(\textbf{x}) : \textbf{x} \in I_s^{2(N-2)} \rbr \\
&\leqslant \sum_{m = 1}^{N-2} \cos \left( \frac{\pi-s}{2} \right)^2 - \prod_{\ell = 1}^{N-2} \sin \left( \frac{\pi - s}{2} \right)^2 \\
&= (N-2) \cdot \sin \left( \frac{s}{2} \right)^2 - \cos \left( \frac{s}{2} \right)^{2(N-2)} \leqslant - \sin \left( \frac{s}{2} \right),
\end{align*}
where the last inequality holds by assumption. Then, since
\begin{align*}
-\sin \left(\frac{s}{2} \right) = \cos \left( \frac{\pi + s}{2} \right),
\end{align*}
taking arccos on both ends of the above inequality gives the first claim. \par 
As for the second claim, we have an analogous calculation:
\begin{align*}
\cos \left( \inf_{x \in I} \alpha_{N-1,N} \lp x_1, \ldots, x_{2(N-2)}, x \rp \right) 
&= F\lp x_1, \ldots, x_{2(N-2)} \rp + X \lp x_1, \ldots, x_{2(N-2)} \rp \\
&\geqslant \inf \lbr F(\textbf{x}) : \textbf{x} \in I_s^{2(N-2)} \rbr + \inf \lbr X(\textbf{x}) : \textbf{x} \in I_s^{2(N-2)}  \rbr \\
&\geqslant -\sum_{m = 1}^{N-2} \cos \left( \frac{\pi - s}{2} \right)^2 + \prod_{\ell = 1}^{N-2} \sin \left( \frac{\pi - s}{2} \right)^2 \\
&= -(N-2) \cdot \sin \left( \frac{s}{2} \right)^2 + \cos \left( \frac{s}{2} \right)^{2(N-2)} \geqslant \sin \left( \frac{s}{2} \right),
\end{align*}
and by the same reasoning as before, this proves the claim. 
\end{proof}
For all $N$ we have $E(N-2,0) = -1$, whereas $E(N-2, \pi) = N-1 > 0$. Hence, by continuity of $s \longmapsto E(N-2, s)$, the number
\begin{align*}
s_N := \inf \left\lbrace s > 0 : E(N-2,s) = 0 \right\rbrace
\end{align*}
exists and is positive. We note that Lemma \ref{p1-lemma3.4} implies that all numbers in the interval $\left(0, s_N \right]$ are $N$-good. \par 
If we let $s_N^{(0)} := s_N$ and define
\begin{align*}
G(N,j,x,s) := (N-j-2) \sin \left( \frac{s}{2} \right)^2 - \sin \left( \frac{x}{2} \right) \cos \left( \frac{s}{2} \right)^{2(N-j-2)} + \sin \left( \frac{s}{2} \right),
\end{align*}
then we can recursively define a sequence $s_N^{(0)}$, $s_N^{(1)}$, $s_N^{(2)}$, $\ldots$, $s_N^{(N-3)}$ by
\begin{align*}
s_N^{(j)} = \inf \left\lbrace s > 0 : G\left(N,j,s_N^{(j-1)}, s \right) = 0 \right\rbrace > 0.
\end{align*}
The existence and positivity of this number follows from considerations analogous to those that ensured these properties for $s_N$. \par
For $N \geqslant 4$ and $1 \leqslant j \leqslant N-3$, say that $s \in (0, \tfrac{\pi}{2})$ is \textit{\textbf{(N, \hspace{0.02em}j)-good}} if, for any $x_1, \ldots, x_{2(N-j-2)} \in I_s$, we have the inclusion
\begin{align*}
I_s \subset \alpha_{N-j-1,N-j} \left( \left( \prod_{m = 1}^{2(N-j-2)} \lbrace x_m \rbrace \right) \times I_{s_N^{(j-1)}} \right).
\end{align*}
Analogously to Lemma \ref{p1-lemma3.4}, we have the following criterion for being $(N,j)$-good. 
\begin{lemma}\label{p1-lemma3.5}
\textit{In order for $s$ to be $(N,j)$-good, it is sufficient that}
\begin{align*}
G\left(N,j, s^{(j-1)}_N, s \right) \leqslant 0.
\end{align*}
\end{lemma}
\begin{proof}
As in the proof of Lemma \ref{p1-lemma3.4}, it is enough to realize that any $s$ that satisfies the inequality has the two properties 
\begin{align*}
\sup_{x \in I_{S_N^{(j-1)}}} \alpha_{N-j-1,N-j} \lp x_1, \ldots, x_{2(N-j-2)}, x \rp &\geqslant \max I_s = \frac{\pi + s}{2}, 
\\ \inf_{x \in I_{S_N^{(j-1)}}} \alpha_{N-j-1,N-j} \lp x_1, \ldots, x_{2(N-j-2)}, x \rp &\leqslant \min I_s = \frac{\pi - s}{2},
\end{align*}
for arbitrary $x_1, \ldots, x_{2(N-j-2)} \in I_s$. The proof goes by estimating the functions $F$ and $X$ exactly as in the proof of Lemma \ref{p1-lemma3.4}, using also that
\begin{align*}
&\cos \left( \sup_{x \in I} \alpha_{N-j-1,N-j}\lp x_1, \ldots, x_{2(N-j-2)}, x \rp \right) \\
&\hspace{1cm} = F \lp x_1, \ldots, x_{2(N-j-2)} \rp + \cos \left( \frac{\pi + s_N^{(j-1)}}{2} \right) X \lp x_1, \ldots, x_{2(N-j-2)} \rp 
\\ &\hspace{1cm} = F \lp x_1, \ldots, x_{2(N-j-2)} \rp - \sin \left( \frac{s_N^{(j-1)}}{2} \right) X \lp x_1, \ldots, x_{2(N-j-2)} \rp 
\\ &\hspace{1cm} \leqslant \sup \lbr F(\textbf{x}) : \textbf{x} \in I_s^{2(N-j-2)}  \rbr - \sin\left( \frac{s_N^{(j-1)}}{2} \right) \cdot \inf \lbr X(\textbf{x}) : \textbf{x} \in I_s^{2(N-j-2)}  \rbr,
\end{align*}
and an analogous inequality for $\cos \left( \inf \alpha_{N-j-1, N-j} \right)$. 
\end{proof}
As before, we see that all the numbers $s$ with $0 < s \leqslant s_N^{(j)}$ are $(N,j)$-good. We now show that for $j > 0$, the property of being $(N,j)$-good is a refinement of the property of being $N$-good. This fact is the main ingredient in the proof of the important Lemma \ref{p1-lemma3.7} below.
\begin{lemma}\label{p1-lemma3.6}
\textit{For a fixed $N \geqslant 3$, the function $j \longmapsto s_N^{(j)}$ is decreasing. Therefore, for any $1 \leqslant j \leqslant N-3$, $\smash{s_N^{(j)}}$ is $(N,j)$-good, $(N,j-1)$-good, $\ldots$, $(N,1)$-good, and $N$-good. }
\end{lemma}
\begin{proof}
Assume that, for some $j$, we have $s_N^{(j)} > s_N^{(j-1)}$. Then, by definition of the sequence $\smash{s_N^{(j)}}$ and the function $G$, we must have
\begin{align*}
G\left(N,j, s_N^{(j-1)}, s_N^{(j-1)}\right) < 0.
\end{align*}
Since we certainly have $0 < \smash{s_N^{(j-1)}} < 2 \pi$, this inequality is impossible since, for $s \in (0, 2 \pi)$, $G(N,j, s,s)$ is negative if and only if
\begin{align*}
(N-j-2) \sin \left( \frac{s}{2} \right) - \cos \left( \frac{s}{2} \right)^{2(N-j-2)} + 1 < 0,
\end{align*}
which cannot hold, given that the left-hand side is at least $(N-j-2) \sin \left( \tfrac{s}{2} \right)$, which is non-negative. This contradiction proves the lemma.
\end{proof}
We can now state and prove the main lemma needed in the proof of Proposition \ref{p1-Prop3.3}.
\begin{lemma}\label{p1-lemma3.7}
\textit{If} $s = s_N^{(N-3)}$\textit{, then $I_s^{b(N)} \subset \Omega_N.$}
\end{lemma}
\begin{proof}
Take any $\lp x_{12}, \ldots, x_{N-1,N} \rp \in I_s^{b(N)}$. We have to show that there are $\phi$'s in $(0, \pi)$ such that, for all $i$ and $j$,
\begin{align*}
\alpha_{ij}(\phi_{1i}, \phi_{1j}, \ldots, \phi_{i-1,i}, \phi_{i-1,j}, \phi_{ij}) = x_{ij}.
\end{align*}
We can obviously take $\phi_{1j} = x_{1j}$ for $j = 2, \ldots, N$ since $I_s \subset (0, \pi )$.\par 
The inductive nature of the method of proof in the cases $i = 2, 3, \ldots, N-1$ should be clear from the concrete instances $i = 2$ and $i = 3$ that follow. If $i = 2$, then the task is to show that for $2 < j \leqslant N$, there is a $\phi_{2j} \in I$ such that $\alpha_{2j} (\phi_{12}, \phi_{1j}, \phi_{2j}) = x_{2j}$. Since $s$ is $(N,N-3)$-good, we know that such a choice of $\phi_{2j}$ exists and is in $I_{s_N^{(N-4)}}$. Hence
\begin{align*}
\phi_{12}, \ldots, \phi_{1N}, \phi_{23}, \ldots, \phi_{2N} \in I_{s_N^{(N-4)}}
\end{align*}
by Lemma \ref{p1-lemma3.6}. \par 
Next, if $i = 3$, then we have to show that for $3 < j \leqslant N$ there is some $\phi_{3j} \in I$ such that $\alpha_{3j}(\phi_{13}, \phi_{1j}, \phi_{23}, \phi_{2j}, \phi_{3j}) = x_{3j}$. Since $s$ is $(N,N-4)$-good by Lemma \ref{p1-lemma3.6}, there is a choice of $\phi_{3j} \in \smash{I_{s_N^{(N-5)}}}$ with this property. Again, by using Lemma \ref{p1-lemma3.6}, we see that 
\begin{align*}
\phi_{12}, \ldots, \phi_{1N}, \phi_{23}, \ldots, \phi_{2N}, \phi_{34}, \ldots, \phi_{3N} \in I_{s_N^{(N-5)}}.
\end{align*}
\par Proceeding by induction, we get
\begin{align*}
\phi_{12}, \ldots, \phi_{1N}, \phi_{23}, \ldots, \phi_{2N}, \phi_{34}, \ldots, \phi_{3N}, \ldots,  \phi_{N-2,N} \in I_{s_N^{(N-(N-2)-2)}} = I_{s_N^{(0)}} = I_{s_N}.
\end{align*}
Finally, we want to find $\phi_{N-1,N} \in I$ such that 
\begin{align*}
\alpha_{N-1,N}\bigl( \phi_{1,N-1}, \phi_{1,N}, \ldots, \phi_{N-2,N-1}, \phi_{N-2,N}, \phi_{N-1,N} \bigr) = x_{N-1,N}.
\end{align*}
However, as $s_N$ is $N$-good, such a $\phi_{N-1,N} \in I$ exists. This concludes the proof. 
\end{proof}
In view of the claim made in Lemma \ref{p1-lemma3.7}, in order to prove (\ref{p1-translationtoN}) it remains to investigate the asymptotic behaviour of $s_N^{(N-3)}$. We have the following result.
\begin{prop}\label{p1-prop3.8}
\textit{For each sufficiently large $N$, there exists an $(N,N-3)$-good number $x$ satisfying $x \geqslant N^{-2.9}$.}
\end{prop}
\noindent \textsc{Remark.} The only important property of the exponent $-2.9$ is that it is strictly bigger than $-3$. In particular, the proposition implies that the sequence $\lbr s_N^{(N-3)} \rbr_N $ cannot decay faster asymptotically than the sequence $N^{-3}$. \\ \par 
In order to prove Proposition \ref{p1-prop3.8}, we need a simple condition that guarantees that a number is $(N,j)$-good. To this end, we require the following lemma, which allows us to estimate the function $G$ from above on small intervals on the positive real line.
\begin{lemma}\label{p1-lemma3.9}
\textit{For $0 \leqslant s \leqslant 2/\sqrt{N-2}$ with $N$ sufficiently large, we have $\sin s \geqslant \varepsilon(N) s$ where $\varepsilon (N) := N^{-9/10(N-3)} e^{1/N}$.}
\end{lemma}
\begin{proof}
Initially, let $0 < \varepsilon < 1$. We want to get a lower bound on the first positive zero of $s \longmapsto \sin s - \varepsilon s$. By Taylor expansions we see that for all $s$, 
\begin{align*}
\cos s \leqslant 1 - \frac{s^2}{2} + \frac{s^4}{24}.
\end{align*}
Along with the Pythagorean theorem, this implies that for $0 \leqslant s \leqslant \pi / 2$,
\begin{align*}
\sin s \geqslant s \sqrt{1 - \frac{s^2}{3} + \frac{s^4}{24}-\frac{s^6}{576}}.
\end{align*} 
Hence for such $s$, we have 
\begin{align*}
\sin s - \varepsilon s \geqslant s \left(-\varepsilon + \sqrt{1 - \frac{s^2}{3} + \frac{s^4}{24}-\frac{s^6}{576}} \right).
\end{align*}
Here, both the left- and right-hand sides are non-negative for small positive values of $s$, so the first positive zero of $\sin s - \varepsilon s$ is at least as big as the first positive zero, $R$, of the right-hand side. We note that $R$ is a solution to the equation
\begin{align*}
s^6 - 24s^4 + 192s^2 - 576 \lp 1-\varepsilon^2 \rp = 0.
\end{align*}
Writing $y = s^2$, we get the cubic equation 
\begin{align*}
y^3 - 24 y^2 + 192 y - 576 \lp 1 - \varepsilon^2 \rp = 0
\end{align*}
whose discriminant is seen to be negative for all choices of $\varepsilon$ in $(0, 1)$. It follows that the cubic has a single real root, which is $4 \lp \sqrt[3]{1-9\varepsilon^2} + 2 \rp >0$. This proves that $\sin s \geqslant \varepsilon s$ for all 
\begin{align*}
0 \leqslant s \leqslant 2 \sqrt{\sqrt[3]{1-9 \varepsilon^2}+2}.
\end{align*} 
\par All that remains is to prove that for the explicit choice of $\varepsilon = \varepsilon(N)$ in the statement of the lemma, we have
\begin{align*}
\frac{1}{\sqrt{N-2}} < \sqrt{\sqrt[3]{1 - 9 \varepsilon(N)^2}+2}.
\end{align*}
The quotient of the right- and left-hand sides is
\begin{align}\label{p1-quotientfromhell}
\sqrt{N-2}\sqrt{\sqrt[3]{1 - 9 \varepsilon(N)^2}+2} \sim \sqrt{N} \cdot \sqrt{\frac{27 \log N}{20(N-3)}}.
\end{align} 
Hence, for $N$ sufficiently large, the quotient in (\ref{p1-quotientfromhell}) is bigger than $1$. This proves the lemma. 
\end{proof}
We now apply Lemma \ref{p1-lemma3.9} to give an explicit condition for a number to be $N$-good or to be $(N,j)$-good for any $j = 1, \ldots, N-3$. 
\begin{lemma}\label{p1-lemma3.10}
\textit{Let $N$ be sufficiently large. In order for $s$ to be $N$-good, it is sufficient that }
\begin{align*}
s \leqslant \frac{1}{N-2} \left( -\frac{1}{2} + \sqrt{\frac{1}{4} + 2(N-2)} \right).
\end{align*}
\textit{In order for $s$ to be $(N,j)$-good, it is sufficient that} $s \leqslant S\left(N,j, s_N^{(j-1)} \right) $ where
\begin{align}\label{p1-incfct}
S\left(N,j, s_N^{(j-1)} \right) := \frac{2 \left( -1 + \sqrt{1 + \varepsilon(N) s_N^{(j-1)}\left(\varepsilon(N) s_N^{(j-1)}+2\right)(N-j-2)} \right)}{\left(\varepsilon(N) s_N^{(j-1)}+ 2\right)(N-j-2)} .
\end{align}
\textit{Furthermore, $S\left(N,j, x \right)$ is increasing in $x$.}
\end{lemma}
\begin{proof}
For any integer $a \geqslant 1$ and $x \in \R$, we have the elementary inequality
\begin{align}\label{p1-cosineq}
\lp \cos x \rp^{a} \geqslant 1 - \frac{ax^2}{2},
\end{align}
which may be proved by inspecting the derivative of the difference between the left- and right-hand side. Using this estimate with $x = \tfrac{s}{2}$, we obtain the first claim from Lemma \ref{p1-lemma3.4}, the fact that $E(N-2,0) < 0$, and the estimate
\begin{align*}
E(N-2,s) &= (N-2)\sin \left(\frac{s}{2} \right)^{2} - \cos \left( \frac{s}{2} \right)^{ 2(N-2)} + \sin  \left(\frac{s}{2} \right) \\ 
&\leqslant (N-2) \frac{s^2}{4} + (N-2) \frac{s^2}{4} - 1 + \frac{s}{2} \\ 
&= \frac{N-2}{2}s^2 + \frac{1}{2}s - 1,
\end{align*}
where the right-hand side is also negative for $s = 0$. The positive root of this quadratic is exactly the claimed upper bound for $s$. \par 
Using (\ref{p1-cosineq}), we prove the second claim as follows: For $0 \leqslant s < 2/\sqrt{N-2}$ we have $(N-j-2)s^2/4 < 1$, so Lemma \ref{p1-lemma3.9} implies that for $s$ in this interval,
\begin{align*}
G\left(N,j, s_N^{(j-1)}, s \right) &= (N-j-2) \sin \left(\frac{s}{2} \right)^{ 2} - \sin \left(\frac{s_N^{(j-1)}}{2} \right) \cos \left(\frac{s}{2} \right)^{2(N-j-2)} + \sin \left(\frac{s}{2} \right) \\
&\leqslant (N-j-2) \frac{s^2}{4} + \frac{\varepsilon(N) s_N^{(j-1)}}{2} \left( \frac{N-j-2}{4}s^2 - 1 \right)  + \frac{s}{2}\\
&=  \frac{N-j-2}{4} \left( \frac{\varepsilon(N) s_N^{(j-1)}}{2} + 1 \right) s^2 + \frac{s}{2} - \frac{\varepsilon(N) s_N^{(j-1)}}{2}.
\end{align*}
This is also negative for $s = 0$, and the positive root of this quadratic is precisely the claimed upper bound for $s$. Assuming for now that the zero $\smash{s_N^{(j)}}$ of the smaller function is less than $2/\sqrt{N-2}$, we get the second claim. \par 
The function $S$ is increasing in its third argument: Writing $\varepsilon = \varepsilon (N)$, we see that 
\begin{align*}
\frac{\partial S(N,j,x)}{\partial x} = \frac{2\varepsilon \left(\varepsilon x(N-j-2)+2(N-j-2) + \sqrt{1+(N-j-2)\varepsilon x(\varepsilon x+2)} - 1 \right)}{(N-j-2)(\varepsilon x+2)^2 \sqrt{1+(N-j-2)\varepsilon x(\varepsilon x+2)}},
\end{align*}
which is visibly positive as $2(N-j-2) - 1 \geqslant 1$. \par 
It remains to justify the claim about the size of $\smash{s_N^{(j)}}$ relative to $2/\sqrt{N-2}$. By Lemma \ref{p1-lemma3.6}, for any fixed $N$ the function $j \longmapsto \smash{s_N^{(j)}}$ is decreasing. Hence it is enough to show that $s_N \leqslant 2/\sqrt{N-2}$. In order to prove this, it is enough to prove that $E \lp N-2, 2 / \sqrt{N-2} \rp \geqslant 0$ for $N$ large enough. Writing $x = \sqrt{N-2}$, we see from Taylor expansions that
\begin{align*}
E \lp N-2, \frac{2}{x} \rp &= x^2 \sin \lp \frac{1}{x} \rp^2 - \cos \lp \frac{1}{x} \rp^{2x^2} + \sin \lp \frac{1}{x} \rp \\
&= x^2 \lp \frac{1}{x} + O \lp x^{-3} \rp \rp^2 -e^{-1} + O \lp x^{-2} \rp + \frac{1}{x} + O \lp x^{-3} \rp \\[+0.7em]
&= 1 - e^{-1} + O \lp x^{-1} \rp
\end{align*}
as $x \longrightarrow \infty$. This proves the lemma. 
\end{proof}
We are now finally ready to prove Proposition \ref{p1-prop3.8}.\\\\
\textit{Proof of Proposition \ref{p1-prop3.8}.} Using Lemma \ref{p1-lemma3.10}, we make the following observation: The lemma implies that
\begin{align}\label{p1-dont}
s_N = s_N^{(0)} \geqslant \frac{1}{N-2} \left( -\frac{1}{2} + \sqrt{\frac{1}{4} + 2(N-2)} \right) =: x_0,
\end{align}
and hence 
\begin{align*}
s_N^{(1)} \geqslant S(N,1, s_N) \geqslant S(N,1, x_0) =: x_1.
\end{align*}
It follows by induction that
\begin{align*}
s_N^{(j)} \geqslant S\left(N,j, s_N^{(j-1)}\right) \geqslant S(N,j, x_{j-1}) =: x_j
\end{align*}
for all $j \leqslant N-3$. Consequently, for any $i$, $x_i$ is $(N,i)$-good. Hence, to prove Proposition \ref{p1-prop3.8}, it suffices to prove that $x_{N-3} \geqslant N^{-2.9}$. To this end, we want to prove that for any $j$ and $\alpha \in \left[-2.9, -2 \right]$\footnote[2]{We stress that the only relevance of the number $-2$ is that it is strictly larger than $-2.9$ and strictly smaller than $-\tfrac{1}{2}$. This last condition is necessary in order for the implication (\ref{p1-imp}) to furnish a contradiction to the definition of $x_0$ in (\ref{p1-dont}) under the assumption that $x_{N-3} < N^{-2.9}$.},
\begin{align}\label{p1-imp}
x_j < N^\alpha \Longrightarrow x_{j-1} < N^{\alpha + 9/10(N-3)}.
\end{align}
The relevance of this is the following. Since
\begin{align*}
-2.9 \leqslant -2.9 + \frac{9j}{10(N-3)} \leqslant -2
\end{align*}
for $j = 0, \ldots, N-3$, it follows by repeated applications of (\ref{p1-imp}) that if we have $x_{N-3} < N^{-2.9}$, then also $x_0 < N^{-2}$. This will contradict (\ref{p1-dont}) and hence Lemma \ref{p1-lemma3.10}, and this contradiction will prove the proposition. Therefore, our goal is now to prove (\ref{p1-imp}).\par 
Let $\alpha \in \left[ -2.9, -2 \right]$ and write $x = \varepsilon (N) x_{j-1}$ and $v = v(N,j) = N-j-2$ to ease the notation. If, for some $j$, we have $x_j < N^\alpha$, then 
\begin{align*}
N^\alpha > \frac{2}{(x + 2)v(N,j)} \left( -1 + \sqrt{1 + x(x+2)v(N,j)} \right).
\end{align*}
This inequality then implies that
\begin{align*}
\left( v N^\alpha \left( \frac{1}{2}x + 1 \right) + 1 \right)^2 > 1 + x(x+2)v = vx^2 + 2vx + 1,
\end{align*}
and therefore
\begin{align*}
\left( \frac{v^2 N^{2 \alpha}}{4} - v \right) x^2 + \left( \lp vN^\alpha + 1 \rp vN^\alpha  - 2v \right) x + v^2 N^{2 \alpha} + 2v N^\alpha > 0.
\end{align*}
Since the leading exponent of this quadratic is negative for $N$ large, the above inequality forces the quadratic to have two distinct roots and $x$ to lie between them. These roots are 
\begin{align*}
r_{\pm} &= \frac{2}{v^2 N^{2\alpha}-4v} \Bigg( 2v - \lp vN^\alpha + 1 \rp vN^\alpha 
\\ &\quad \quad \left. \pm \sqrt{\bigl( 2v - \lp vN^\alpha + 1 \rp vN^\alpha \bigr)^2 - \lp v^2 N^{2 \alpha} + 2vN^\alpha \rp \lp v^2 N^{2 \alpha} - 4v\rp }  \right).
\end{align*}
In particular, we see that $x$ must be less than
\begin{align}\label{p1-goodtoknow}
\nonumber r_- &= \frac{2}{v^2 N^{2\alpha}-4v} \Bigg( 2v - \lp vN^\alpha + 1 \rp vN^\alpha - \sqrt{\bigl( 2v - \lp vN^\alpha + 1 \rp vN^\alpha \bigr)^2 - \lp v^2 N^{2 \alpha} + 2vN^\alpha \rp \lp v^2 N^{2 \alpha} - 4v\rp }  \Bigg) 
\\ \nonumber &= -2 + \frac{4+2N^\alpha}{4-vN^{2 \alpha}} + \frac{2\sqrt{4 + (vN^\alpha + 1)^2 N^{2 \alpha} - 4N^\alpha (v N^\alpha + 1) - v^2 N^{4 \alpha} + 4 v N^{2 \alpha} - 2v N^{3 \alpha} + 8N^\alpha}}{4-vN^{2 \alpha}}
\\[+0.3em] \nonumber &= -2 + \frac{4+2N^\alpha}{4-vN^{2 \alpha}} + \frac{2\sqrt{ 4 + 4N^\alpha + N^{2 \alpha}}}{4-vN^{2 \alpha}} 
\\[+0.3em] &= -2 + \frac{8 + 4N^\alpha}{4 - v N^{2 \alpha}}.
\end{align}
We wish to compare $r_-$ to $N^\alpha$. We observe that $r_-$ is as large as possible when $j$ is as small as possible (since $4-vN^{2 \alpha}$ is increasing in $j$), meaning that it will be enough to estimate $r_-$ for $j = 0$. In this case $v = N-2$, whence
\begin{align}\label{p1-almost}
r_- &= \frac{4N^\alpha - 4N^{2 \alpha} + 2N^{2 \alpha + 1}}{4 + 2N^{2 \alpha} - N^{2 \alpha + 1}} <  \frac{N^\alpha \left( 1 + \tfrac{1}{2} N^{\alpha + 1} \right)}{1 + \tfrac{1}{2}N^{2 \alpha} - \tfrac{1}{4}N^{2 \alpha + 1}} < \frac{N^\alpha \left( 1 + \tfrac{1}{2} N^{\alpha + 1} \right)}{1 - \tfrac{1}{4}N^{2 \alpha + 1}} = N^\alpha Z(N, \alpha),
\end{align}
where we put
\begin{align*}
Z(N, \alpha) := \frac{1 + \tfrac{1}{2} N^{\alpha + 1}}{1 - \tfrac{1}{4}N^{2 \alpha + 1}}.
\end{align*}
Since $x = \varepsilon (N) x_{j-1}$, it now follows from (\ref{p1-goodtoknow}) and (\ref{p1-almost}) that
\begin{align*}
x_{j-1} < N^\alpha \varepsilon(N)^{-1} Z(N, \alpha).
\end{align*}
Next, we obtain
\begin{align*}
\log \left( Z(N, \alpha)^{10(N-3)/9} \right) &= \frac{10}{9}(N-3) \left( \log \left( 1 + \frac{1}{2}N^{\alpha + 1} \right) - \log \left( 1 - \frac{1}{4} N^{2 \alpha + 1} \right) \right)\\
&= \frac{10}{9}(N-3) \left( \frac{1}{2} N^{\alpha + 1} + O\left( N^{2 \alpha + 2} \right) \right) \\ &= \frac{5}{9}N^{\alpha + 2} + O\left(N^{\alpha + 1} \right).
\end{align*}
Furthermore, since
\begin{align*}
\log \varepsilon(N)^{-1} = \log \left( e^{-1/N} N^{9/10(N-3)} \right) = -\frac{1}{N} + \frac{9 \log N}{10(N-3)},
\end{align*}
when $N$ is large enough (independently of $\alpha$) we have
\begin{align*}
\log \left(\left( Z(N, \alpha) \varepsilon(N)^{-1}\right)^{10(N-3)/9} \right) &= \frac{5}{9} N^{\alpha + 2} + \frac{10}{9} (N-3) \left( -\frac{1}{N}+ \frac{9 \log N}{10(N-3)} \right) + O\left(N^{\alpha + 1} \right)
\\[+0.5em] &= \frac{5}{9} N^{\alpha + 2} - \frac{10(N-3)}{9N} + \log N + O\left(N^{\alpha + 1} \right) 
\\[+0.5em] &\leqslant \log N.
\end{align*}
Hence $Z(N, \alpha)\varepsilon(N)^{-1} < N^{9/10(N-3)}$, and it follows that for $N$ large enough (independently of $\alpha$), one has
\begin{align*}
x_{j-1} < N^{\alpha + 9/10(N-3)}.
\end{align*}
This proves the implication (\ref{p1-imp}), and the proof of Proposition \ref{p1-prop3.8} is therefore concluded. $\hfill \blacksquare$\\\\
\end{appendices}
This completes the proof of Proposition \ref{p1-Prop3.3}.\\[+1em] \par 
\textsc{Acknowledgements.} The author is very grateful to Anders Södergren for suggesting the problem, and for many helpful discussions.\\\\
\renewcommand{\section}[2]{}%
\renewcommand{\refname}{REFERENCES}

\end{document}